\documentclass[11pt]{article}
\usepackage{amsfonts,amsmath,amssymb,verbatim}
\usepackage{amsthm,comment}
\usepackage{graphicx,subfigure,psfrag,amscd,color}
\usepackage{color}
\usepackage[]{appendix}
\newtheorem{thm}{\bf Theorem}[section]
\newtheorem{prop}[thm]{\bf Proposition}
\newtheorem{lem}[thm]{\bf Lemma}
\newtheorem{cor}[thm]{\bf Corollary}

\newtheorem{definition}[thm]{\bf Definition}
\renewcommand{\labelenumi}{(\roman{enumi})}
\def \R {\mathbb{R}}
\def \C {\mathbb{C}}

\def \id {\operatorname{id}}
\def \Ad {\operatorname{Ad}}
\def \ad {\operatorname{ad}}
\def \dim {\operatorname{dim}}

\def \o {\mathrm{o}}
\def \g {\mathfrak{g}}
\def \h {\mathfrak{h}}
\def \j {\mathfrak{j}}
\def \l {\mathfrak{l}}
\def \p {\mathfrak{p}}
\def \n {\mathfrak{n}}
\def \G {\mathcal{G}}
\def \F {\mathcal{F}}
\def \P {\mathrm{Poly}}
\def \a {\alpha}

\begin{document}
\title{Local rigidity of certain actions of solvable groups on the boundaries of rank one symmetric spaces}
\author{Mao Okada}
\maketitle
\begin{abstract}
Let $G$ be the group of orientation-preserving isometries of a rank-one symmetric space $X$ of non-compact type.
We study local rigidity of certain actions of a solvable subgroup $\Gamma \subset G$ on the boundary of $X$, which is diffeomorphic to a sphere.
When $X$ is  a quaternionic hyperbolic space or the Cayley hyperplane, the action we constructed is locally rigid.

\end{abstract}
\section{Introduction}
One of the most active areas of the study of rigidity of group actions is around the Zimmer program, in which many remarkable properties of actions of a lattice $\Gamma$ of a higher rank Lie group have been discovered.
See \cite{Fisher} for the recent development.
As pointed out in \cite{Fisher}, in the study of actions of a lattice of a higher rank Lie group, the study of actions of a higher rank abelian group $\Gamma = \mathbb{Z}^n$, $n \geq2$ of certain hyperbolicity plays an important role.
On the other hand, Burslem and Wilkinson showed that there exists a solvable group $\Gamma$ which does not contain a higher rank abelian group such that an action on the circle $\mathrm{S}^1$ is locally rigid \cite{Burslem-Wilkinson}.
In this paper, we consider locally rigid actions of solvable groups which does not contain a hyperbolic action of a higher rank abelian group.

As a higher dimensional analogue of the result of Burslem and Wilkinson, Asaoka constructed an action of a solvable group on $\mathrm{S}^n$, $n \geq 2$ \cite{Asaoka1}.
Asaoka showed that, while the action is not locally rigid, it is locally rigid in a weaker sense.
One of the most important example of such a weak form of local rigidity is \cite{Ghys}.
In \cite{Asaoka2}, Asaoka studied local rigidity of an action of the same group on the torus $\mathrm{T}^n$, which can also be viewed as a higher dimensional version of the result of Burslem and Wilkinson.
In \cite{Okada}, the author studied local rigidity of certain action of a solvable group on the sphere.
In \cite{Wilkinson-Xue}, Wilkinson and Xue studied rigidity of an action of a solvable group on the torus.

In this paper, we consider a generalization of the results of \cite{Asaoka1} and \cite{Okada} which can be formulated as follows.
Let $X$ be a rank one symmetric space of non-compact type, $G$ the group of orientation-preserving isometries of $X$, and $G = KAN$ an Iwasawa decomposition.
\begin{definition}
A subgroup $\Gamma$ of $AN \subset G$ is called a \textrm{standard subgroup} of $G = KAN$ if $\ \Gamma$ is generated by a lattice $\Lambda$ of $N$ and a nontrivial element $a \in A$ such that $a\Lambda a^{-1} \subset \Lambda$.
\end{definition}
Let $M \subset K$ be a centralizer of $A$ in $K$ so that $P = MAN$ is a minimal parabolic subgroup of $G$.
Then the homogeneous space $G/P$ is diffeomorphic to a sphere.
The action of $G$ on $G/P$ by the left translation will be denoted by $l: G \to \mathrm{Diff}(G/P)$.
The following theorem, which can be referred to as $C^2$-local rigidity of $l|_\Gamma$ up to embedding of $\Gamma$ into $G$, is the main theorem of this paper.
\begin{thm}\label{thm:main}
Let $G$ be the group of orientation-preserving isometries of a rank one symmetric space of non-compact type, $\Gamma$ a standard subgroup of $G$, and $l|_\Gamma$ the action of $\,\Gamma$ on $G/P$ by left translations.
Assume $G \neq \mathrm{PSL}(2, \R)$.
If $\rho$ is a $C^\infty$ action of $\,\Gamma$ on $G/P$ sufficiently $C^2$-close to $l|_\Gamma$, then there is an embedding $\iota$ of $\,\Gamma$ into $G$ as a standard subgroup and a $C^\infty$ diffeomorphism $h$ of $G/P$ such that 
\[
\rho(g) = h \circ l(\iota(g)) \circ h^{-1}
\]
for all $g \in \Gamma$.
\end{thm}

While we excluded the case $G = \mathrm{PSL}(2, \R)$ for a technical reason, the claim also holds.
In this case, $\Gamma$ can be presented as 
\[
\langle a, b \mid aba^{-1} = b^k \rangle
\]
for some integer $k \geq 2$ and $G/P$ is diffeomorphic to a circle $S^1$.
The action $l|_\Gamma$ admits a common fixed point and the action on the complement, which is diffeomorphic to $\R$, is given by
\[
a\cdot x = kx,\; b\cdot x = x + 1\;\; (x \in \R).
\]
The local rigidity of the action follows from the result of Burslem and Wilkinson mentioned above.
It is not difficult to check that the case $G = \mathrm{SO}_0(n+1, 1)$, $n \geq 2$ is exactly the above result of Asaoka.
The case $G = \mathrm{SU}(n+1, 1)$, $n \geq 2$ for $C^3$-small perturbation is the above result of the author.

When $G = \mathrm{Sp}(n+1, 1)$, $n \geq 2$ or $F_4^{-20}$, we can show the inclusion $\Gamma \hookrightarrow G$ is locally rigid.
So we obtain local rigidity in the strict sense:
\begin{cor}\label{cor:main}
For $G = \mathrm{Sp}(n+1, 1)$ $(n \geq 2)$ and $F_4^{-20}$, the action $l|_\Gamma$ of a standard subgroup $\Gamma$ of $G$ on $G/P$ is $C^2$-locally rigid;
a $C^\infty$ action sufficiently $C^2$-close to $l|_\Gamma$ is $C^\infty$-conjugate to $l|_\Gamma$.
\end{cor}

It should be pointed out that the action $l|_\Gamma$ is not locally rigid in the remaining cases.
When $G = \mathrm{SO}_0(n+1, 1)$, $n \geq 2$,  the classification up to conjugacy of the actions of standard subgroups by left translations is given in \cite{Asaoka1}.
In particular, the action $l|_\Gamma$ is not locally rigid.
When $G = \mathrm{SU}(n+1, 1)$, $n \geq 2$, we can also show that $l|_\Gamma$ is not locally rigid. See Proposition \ref{prop:nonloc_rigid}.

The proof of the main theorem can be described as follows.
Put $\o = eP \in G/P$.
The point $\o$ is the common fixed point of the action $l|_\Gamma$.
Using Stowe's theorem \cite{Stowe}, we see that an action close to $l|_\Gamma$ also admits a common fixed point close to $\o$.
Moreover, by an argument similar to that of \cite{Asaoka1}, a conjugacy defined around the common fixed points extends to a diffeomorphism of the whole $G/P$.
So the main theorem is reduced to local rigidity of a homomorphism of $\Gamma$ into the group $\G(G/P, \o)$ of germs of diffeomorphisms defined around $\o \in G/P$ and fixing $\o \in G/P$.

Such a problem of classification of  ``local action'' around a fixed point can be found in \cite{Sternberg}.
In fact, Sternberg's theory of normal forms of a hyperbolic fixed point of diffeomorphisms can be viewed as certain rigidity of a homomorphism of $\mathbb{Z}$ into $\G(\R^n, 0)$.
To prove certain local rigidity of a homomorphism of $\Gamma$ into $\G(G/P, \o)$, we will adopt a similar strategy.
Sternberg's theory can be summarized as follows.
The first step is to show that a diffeomorphism around a hyperbolic fixed point is determined by its Taylor expansion at the fixed point.
The next step is to show there exists $r \geq 1$ such that the Taylor expansion is determined by the derivatives of order at most $r$ at the fixed point, where $r$ depends on the ``resonance'' of the eigenvalues of the first-order derivative of the diffeomorphism at the fixed point.
The last step is the classification of elements in  the group $J^r(\R^n, 0)$ of $r$-jets at $0 \in \R^n$ of diffeomorphisms around $0 \in \R^n$.

In our case, the problem can be reduced to local rigidity of a homomorphism of $\Gamma$ into $J^3(G/P, \o)$.
Computing the induced homomorphism of $\Gamma$ into $J^1(G/P, \o)$, and using a theorem of Malcev, we will show that the problem can be reduced to local rigidity of a homomorphism of the closure $\langle a \rangle N$ of $\Gamma$ in $G$ into $J^3(G/P, \o)$.
Then the problem reduces to the computation of the cohomology $H^1(\n, \j^3(G/P, \o))^{\langle a \rangle} = H^1(\n, \j^3(G/P, \o))^{\mathfrak{a}}$, where $\n$, $\j^3(G/P, \o)$, and $\mathfrak{a}$ denote the Lie algebras of $N$, $J^3(G/P, \o)$ and $A$, respectively.
The computation of such a cohomology is, as we can see in \cite{Asaoka1} and \cite{Okada}, one of the most difficult part of the proof.
In this paper, we will compute the cohomology using some tools from (non-unitary) representation theory as well as an explicit classification of rank one simple Lie algebras.
As a result, we obtain an isomorphism
\[
H^1(\n, \j^3(G/P, \o))^{\mathfrak{a}} = H^1(\n, \g)^{\mathfrak{a}},
\]
which means that perturbation of the homomorphism of $\langle a \rangle N$ into $J^3(G/P, \o)$ is locally rigid up to embedding of $\langle a \rangle N$ into $G$.
Moreover, it is not difficult to check $H^1(\n, \g)^{\mathfrak{a}} = 0$ if and only if $G = \mathrm{Sp}(n+1, 1)$ $(n \geq 2)$ or $F_4^{-20}$, in which case our result is in fact local rigidity in the strict sense.

In Section \ref{sect:prel}, we collect facts which will be used later.
In particular, in Subsection \ref{subsect:std_boundary} we establish a fundamental property of the action $l|_\Gamma$ of a standard subgroup $\Gamma$ on $G/P$.
Section \ref{sect:Lie_cohom} is devoted to the computation of the cohomology of $\n$ mentioned above.
In Section \ref{section:vanish}, we compute certain cohomology of a standard subgroup $\Gamma$, vanishing of which is the assumption of the above theorem of Stowe.
In Section \ref{sect:loc_rigid_jet}, we study local rigidity of a homomorphism of a standard subgroup $\Gamma$ into the group $J^3(G/P, \o)$ of 3-jets.
In Section \ref{sect:loc_rigid_formal}, we consider local rigidity of a homomorphism of a standard subgroup $\Gamma$ into the group $\F(G/P, \o)$ of Taylor expansions of the diffeomorphisms in $\G(G/P, \o)$, called the group of formal transformations.
In Section \ref{sect:loc_rigid_loc}, we study local rigidity of a homomorphism of a standard subgroup $\Gamma$ into the group $\G(G/P, \o)$ of germs of diffeomorphism defined around $\o \in G/P$ fixing $\o \in G/P$.
In Section \ref{sect:loc_rigid_act}, we prove the main theorem.

\section*{Acknowledgement}
I would like to thank Hirokazu Maruhashi for suggesting the use of representation theory and Masahiko Kanai for his comments that greatly improved the manuscript.

\section{Preliminaries}\label{sect:prel}
\subsection{Representation of a semisimple Lie algebra}\label{subsect_ss_repr}
The goal of this subsection is to introduce Theorem \ref{thm:Vogan} and Theorem \ref{thm:Kostant}.
See \cite{Vogan} for the detail.
Theorem \ref{thm:Kostant} is the formula for cohomology of the nilradical $\n$ of a parabolic subalgebra $\p$ of a complex semisimple Lie algebra $\g$ with the coefficient in a finite-dimensional $\g$-module, while in Section \ref{sect:Lie_cohom} we have to compute cohomology of $\n$ with the coefficient in an infinite-dimensional $\g$-module.
The proof of Theorem \ref{thm:Kostant} due to Casselman and Osborne contains a study of an infinite-dimensional $\g$-module.
Theorem \ref{thm:Vogan} is a consequence of the result of Casselman and Osborne, the formulation of which is due to Vogan.
We will use Theorem \ref{thm:Vogan} for our computation in Section \ref{sect:Lie_cohom}.

To state Theorem \ref{thm:Vogan}, we review representation theory of semisimple Lie algebras.
A standard reference is \cite{Knapp}.
Let $\g$ be a complex semisimple Lie algebra, $\h \subset \g$ a Cartan subalgebra, and $\Delta(\g, \h)\subset \h^*$ the system of roots.
Fix a positivity on $\h^*$.
Let $\Delta^+(\g, \h) \subset \Delta(\g, \h)$ be the system of positive roots.
Then the subalgebra
\[
\mathfrak{b} = \h \oplus \bigoplus_{\alpha \in \Delta^+(\g, \h)} \g_\alpha
\]
is called the corresponding \textit{Borel subalgebra}.
A subalgebra $\p$ of $\g$ containing $\mathfrak{b}$ is called a \textit{parabolic subalgebra}.
A parabolic subalgebra $\p$ admits the decomposition
\[
\p = \mathfrak{l} \oplus \n
\]
such that
\[
\n = \bigoplus_{\a \in \Delta(\n, \h)} \g_\a
\]
is a nilpotent subalgebra with $\Delta(\n, \h)$ contained in $\Delta^+(\g, \h)$, and
\[
\mathfrak{l} = \h \oplus \bigoplus_{\a \in \Delta(\mathfrak{l}, \h)} \g_\a
\]
is reductive with $\Delta(\mathfrak{l}, \h) =  \Delta(\g, \h)\setminus(\Delta(\n, \h) \cup -\Delta(\n, \h))$, where $-\Delta(\n, \h)= \{ -\a \mid \a \in \Delta(\n, \h)\}$.
The subalgebra
\[
\n_- = \bigoplus_{\a \in -\Delta(\n, \h)} \g_\a
\]
is called  the \textit{opposite} of $\n$ and we obtain the decomposition
\[
\g = \n_- \oplus \mathfrak{l} \oplus \n
\]
of $\g$ as a vector field.

Let $\g$ be a complex semisimple Lie algebra, $U(\g)$ the universal enveloping algebra of $\g$, and $Z(\g)$ the center of $U(\g)$.
There is an isomorphism called the \textit{Harish-Chandra isomorphism} of $Z(\g)$ onto the algebra $U(\h)^W$ of the Weyl group $W = W(\g, \h)$ invariant elements of $U(\h)$ which can be constructed as follows.
Let $\g = \g_- \oplus \h \oplus \g_+$ be the decomposition corresponding to the Borel subalgebra $\mathfrak{b} = \h \oplus \g_+$, where $\g_\pm = \oplus_{\alpha \in \pm\Delta^+(\g, \h)} \g_\alpha$.
By Poincar\'e-Birkhoff-Witt theorem,
\[
U(\g) = U(\h) \oplus (\g_- U(\g) + U(\g)\g_+).
\]
Let $p:U(\g) \to U(\h)$ be the projection onto the first term.
Define the shift map $\sigma_c:U(\h) \to U(\h)$, $c \in \C$ by the extension of $\h \to U(\h)$, $X \mapsto X - cX$ as a homomorphism of algebra.
Then the composition $\gamma = \sigma_{\delta(\g)} \circ p:U(\g) \to U(\h)$ is the Harish-Chandra map, where
\[
\delta(\g) = \frac{1}{2}\sum_{\a \in \Delta^+(\g, \h)} \a
\]
is the \textit{lowest form} of $\g$.
It is known that the Harish-Chandra map induces an isomorphism between $Z(\g)$ and $U(\h)^W$ called the Harish-Chandra isomorphism and that the Harish-Chandra isomorphism does not depend on the choice of a positivity on $\h^*$.

Let $\g$ be a complex semisimple Lie algebra and $Z(\g)$ the center of the universal enveloping algebra $U(\g)$.
A $\g$-module is naturally a module over $U(\g)$.
A representation of $\g$ is said to admit an \textit{infinitesimal character} if each element of $Z(\g)$ acts by multiplication by a scalar.
In this case, the action of $Z(\g)$ is described by a homomorphism of $Z(\g)$ into $\C$.
Via the Harish-Chandra isomorphism, we obtain a homomorphism of $U(\h)^W$ into $\C$.
Since $\h$ is abelian, $U(\h)$ can be considered as the algebra of polynomial functions on $\h^*$.
It is not difficult to see that a homomorphism of $U(\h)^W$ into $\C$ is the evaluation map $\mathrm{ev}_\lambda$ at a point $\lambda \in \h^*$ and $\mathrm{ev}_\lambda = \mathrm{ev}_\mu$ if and only if $\lambda$ and $\mu$ have the same $W$-orbit.
Such a $\lambda$ is called an \textit{infinitesimal character} of the representation.

A typical example of a representation with an infinitesimal character is an irreducible finite-dimensional representation.
A weight vector in a $\g$-module is \textit{highest} (resp. \textit{lowest}) if it is annihilated by $\g_+$ (resp. $\g_-$).
Let $F^\g_\lambda$ of $\g$ be the irreducible finite-dimensional representation with a highest weight vector of weight $\lambda$.
By the construction of the Harish-Chandra isomorphism, we see that it has an infinitesimal character $\lambda + \delta(\g)$.
More generally, let $V$ be a (possibly infinite-dimensional) representation $V$ of $\g$ such that each root space $V_\a$ $(\a \in \h^*)$ is finite dimensional.
Then $V$ admits an infinitesimal character $\lambda + \delta(\g)$ if $V$ has a unique highest weight vector (up to scalar multiplication) of weight $\lambda$.
Similarly, using the fact that the Harish-Chandra isomorphism does not depend on the positivity on $\h^*$, $V$ admits an infinitesimal character $\lambda - \delta(\g)$ if $V$ has a unique lowest weight vector of weight $\lambda$.

Let $\g$ be a complex semisimple Lie algebra with the decomposition $\g = \n_- \oplus \l \oplus \n$ corresponding to a parabolic subalgebra.
A $\g$-module $V$ is $\l$-\textit{finite} if $V$ admits a decomposition into the sum of (possibly infinitely many) finite-dimensional representations of $\mathfrak{l}$.

\begin{thm}[\cite{Vogan} Corollary 3.1.6.]\label{thm:Vogan}
Let $\g$ be a complex semisimple Lie algebra with the decomposition $\g = \n_- \oplus \mathfrak{l} \oplus \n$ and $V$ be a representation of $\g$ which admits an infinitesimal character $\lambda$.
Assume $V$ is $\l$-finite.
Then $H^*(\n, V)$ also admits a decomposition into the sum of finite-dimensional representations of $\mathfrak{l}$.
Moreover, a weight $\mu \in \h^*$ appears as an $\mathfrak{l}$-highest weight of $H^*(\n, V)$ only if $\mu + \delta(\g)$ and $\lambda$ have the same $W$-orbit.
\end{thm}
When $V$ is finite dimensional, using this theorem, one can completely determine $H^*(\n, V)$.
For $w \in W$, the smallest number $\mathrm{length}(w) = n$ such that $w$ is a product of $n$ reflections in simple roots is called the \textit{length} of $w$.

\begin{thm}[Kostant, see \cite{Vogan} Theorem 3.2.3.]\label{thm:Kostant}
Let $\g$ be a complex semisimple Lie algebra with the decomposition $\g = \n_- \oplus \mathfrak{l} \oplus \n$ and $\mathrm{F}^\g_\lambda$ be the irreducible finite-dimensional representation of $\g$ with the highest weight $\lambda$.
Then
\[
H^r(\n, \mathrm{F}^\g_\lambda) = \bigoplus _\mu \mathrm{F}^\l_\mu 
\]
as an $\l$-module, where the sum is taken over $\mu = w(\lambda + \delta(\g)) - \delta(\g)$ for $ w \in W, \,r = \mathrm{length}(w)$.
\end{thm}

\subsection{Classification of the simple Lie algebras of real rank one}\label{subsect:rank_one_Lie}
Let $\g$ be a real simple Lie algebra and $\g = \mathfrak{k} \oplus \mathfrak{a} \oplus \n$ be an Iwasawa decomposition.
The dimension of $\mathfrak{a}$ is called the \textit{real rank} of $\g$.
Assume the real rank of $\g$ is one.
Then there is a restricted-root decomposition 
\[
\g = \bigoplus_{i = -2}^2 \g_i
\]
such that
\[
\n = \g_1 \oplus \g_2,
\]
where $\g_i = \{X \in \g \mid [H, X] = i\a(H)X \mathrm{\;for\;}\mathrm{all\;} H \in \mathfrak{a}\}$ for some $\a \in \mathfrak{a}^*$.
Then the subalgebra
\[
\p =  \bigoplus_{i = 0}^2 \g_i
\]
is a minimal parabolic subalgebra of $\g$.
Let $\g_\C$ be the complexification of $\g$.
There exists a Cartan subalgebra $\h$ of $\g_\C$ such that $\h \subset (\g_0)_\C$.
We fix a positivity on $\h^*$ such that $\Delta(\n_\C, \h) \subset \Delta^+(\g_\C, \h)$.
Then $\p_\C = (\g_0)_\C \oplus \n_\C$ is the decomposition of the parabolic subalgebra $\p_\C$ as in Subsection \ref{subsect_ss_repr}.

By the classification of real simple Lie algebras, $\g_\C$ is isomorphic to $\mathfrak{so}(n, \C)$, $\mathfrak{sl}(n, \C)$, $\mathfrak{sp}(2n, \C)$, or $\mathfrak{f}_4$.
In each case, the system $\Delta(\g, \h)$ of roots can be expressed as follows.
Let $\langle \cdot, \cdot \rangle$ be the bilinear form on $\h^*$ induced by the Killing form on $\h$.

When $\g_\C = \mathfrak{so}(2n+1, \C)$, $n \geq 1$, there is a complex linear basis $e_1, \dots, e_n$ of the dual $\h^*$ of $\h$ with $\langle e_i, e_j \rangle = 0$ for $i\neq j$ and $|e_i|^2 = |e_j|^2$ such that
\begin{align*}
\Delta(\g_\C, \h) &= \{\pm e_i\pm e_j \mid 1 \leq i < j \leq n\} \cup \{\pm e_i\mid 1 \leq i \leq n\},\\
\Delta(\n_\C, \h) &= \{ e_1\pm e_j \mid 2 \leq  j \leq n\} \cup \{e_1\},\\
\Delta((\g_0)_\C, \h) &= \{\pm e_i\pm e_j \mid 2 \leq i < j \leq n\} \cup \{\pm e_i\mid 2 \leq i \leq n\}.
\end{align*}

When $\g_\C = \mathfrak{so}(2n, \C)$, $n \geq 1$, there is a basis $e_1, \dots, e_n$ of $\h^*$ with $\langle e_i, e_j \rangle = 0$ for $i\neq j$ and $|e_i|^2 = |e_j|^2$ such that
\begin{align*}
\Delta(\g_\C, \h) &= \{\pm e_i\pm e_j \mid 1 \leq i < j \leq n\},\\
\Delta(\n_\C, \h) &= \{ e_1\pm e_j \mid 2 \leq  j \leq n\},\\
\Delta((\g_0)_\C, \h) &= \{\pm e_i\pm e_j \mid 2 \leq i < j \leq n\}.
\end{align*}

When $\g_\C = \mathfrak{sl}(n, \C)$, $n \geq 2$, there is an $n$-dimensional vector space with a bilinear form with a basis $e_1, \dots, e_n$ satisfying $\langle e_i, e_j \rangle = 0$ for $i\neq j$ and $|e_i|^2 = |e_j|^2$ such that there is an identification of $\h^*$ with the subspace $\{\sum_i c_ie_i \mid \sum c_i = 0\}$ under which
\begin{align*}
\Delta(\g_\C, \h) &= \{\pm (e_i - e_j) \mid 1 \leq i < j \leq n\},\\
\Delta(\n_\C, \h) &= \{ e_1 - e_j \mid 1 < j \leq n\} \cup \{ e_i - e_n \mid 1 \leq i < n\},\\
\Delta((\g_0)_\C, \h) &= \{\pm (e_i - e_j )\mid 2 \leq i < j \leq n-1\}.
\end{align*}

When $\g_\C = \mathfrak{sp}(2n, \C)$ with $n \geq 3$,
\footnote{
$\mathfrak{sp}(4, \C)$ is isomorphic to $\mathfrak{so}(5, \C)$.
}
there is a basis $e_1, \dots, e_n$ of $\h^*$ with $\langle e_i, e_j \rangle = 0$ for $i\neq j$ and $|e_i|^2 = |e_j|^2$ such that
\begin{align*}
\Delta(\g_\C, \h) &= \{\pm e_i\pm e_j \mid 1 \leq i < j \leq n\} \cup \{\pm2e_i \mid 1 \leq i \leq n \},\\
\Delta(\n_\C, \h) &= \{ e_i\pm e_j \mid 3 \leq  j \leq n, \; i =1, 2\} \cup \{2e_1, e_1 +e_2, 2e_2\},\\
\Delta((\g_0)_\C, \h) &= \{\pm e_i\pm e_j \mid 3 \leq i < j \leq n\} \cup  \{\pm2e_i \mid 3 \leq i \leq n \} \cup \{\pm(e_1 - e_2)\}.
\end{align*}

When $\g_\C = \mathfrak{f}_4$, there is a basis $e_1, \dots, e_4$ of $\h^*$ with $\langle e_i, e_j \rangle = 0$ for $i\neq j$ and $|e_i|^2 = |e_j|^2$ such that
\begin{align*}
\Delta(\g_\C, \h) &= \{\pm e_1 \pm e_2 \pm e_3 \pm e_4\}\cup\{\pm 2e_i \pm 2e_j \mid 1 \leq i < j \leq 4\}\\
&\:\:\:\:\:\: \cup \{\pm2e_i \mid 1\leq i \leq 4\},\\
\Delta(\n_\C, \h) &= \{e_1 \pm e_2 \pm e_3 \pm e_4\}\cup\{ 2e_1 \pm 2e_j \mid 2 \leq j \leq 4\} \cup \{2e_1\},\\
\Delta((\g_0)_\C, \h) &= \{\pm 2e_i \pm 2e_j \mid 2 \leq i < j \leq 4\} \cup \{\pm2e_i \mid 2\leq i \leq 4\}.
\end{align*}
\subsection{Vector fields on a vector space}\label{subsect:vfvs}
Let $V$ be a finite-dimensional vector field over $\R$.
At each point $x$ of $V$, there is a natural identification of the tangent space $T_xV$ with $V$ itself.
Let $S(V) = \bigoplus_{r \geq 0} S^r(V)$ be the space of symmetric tensor products of $V$.
Consider the space $S(V^*) \otimes V$ of $V$-valued polynomial functions on $V$, where $V^*$ is the dual of $V$.
For each $f \in S(V^*) \otimes V$, we define the vector field $X_f$ on $V$ by
\[
X_f(x) = -f(x) \in V = T_xV
\]
for $ x \in V$.
Such a vector field will be called a \textit{polynomial vector field} on $V$.
A polynomial vector field corresponding to a constant function $v \in V \subset S(V^*) \otimes V$ will be called a \textit{constant vector field}.
Observe that a smooth vector field $X$ on $V$ is polynomial if and only if there exist $r \geq 0$ such that $\ad(X_1)\dots \ad(X_r)X = 0$ for any constant vector fields $X_1, \dots X_r$ on $V$. 
The Lie algebra of polynomial vector fields will be denoted by $\P(V)$.
We identify $\P(V)$ with $S(V^*) \otimes V$ by the above equation.
Under this identification,
\[
[S^p(V^*) \otimes V, S^q(V^*) \otimes V] \subset S^{p+q-1}(V^*) \otimes V
\]
for $p, q \geq 0$.
So $\P(V) = \bigoplus_{r \geq 0} S^r(V^*) \otimes V$ is naturally a graded Lie algebra.
It is  not difficult to check that if $V$ is a representation of a group $G$, then this identification is an isomorphism between $G$-modules.

The grading on $\P(V) $ is convenient to describe the structure of the group of jets.
For $r\geq 1$, let $J^r(V, 0)$ be the group of $r$-jets at $0 \in V$ of the diffeomorphism defined around $0$ and fixing $0$.
Then its Lie algebra $\j^r(V, 0)$ is naturally a quotient of the Lie algebra $\P(V, 0) = \bigoplus_{r \geq 1} S^r(V^*) \otimes V$ of polynomial vector fields vanishing at $0 \in V$.
In fact,
\[
\j^r(V, 0) = \P(V, 0)/ \bigoplus_{q \geq r+1} S^{q}(V^*) \otimes V.
\]
Thus $\j^r(V, 0)$ can be identified with $\bigoplus_{1 \leq q \leq r} S^{q}(V^*) \otimes V$ as a linear space.
When $V$ is a representation of a group $G$, this identification is an isomorphism between $G$-modules.

\subsection{The standard actions on the boundaries of rank one symmetric spaces}\label{subsect:std_boundary}
\subsubsection{The boundaries of rank one symmetric spaces}
Let $X$ be a rank one symmetric space of non-compact type and $G$ the group of orientation-preserving isometries of $X$, 
Then $X = G/K$, where $K$ is a maximal compact subgroup of $G$.
Fix an Iwasawa decomposition $G = KAN$.
Let $M = \{k \in K\mid ak = ka \;\mathrm{for\; all}\; a \in A\}$ be the centralizer of $A$ in $K$ so that $P = MAN \subset G$ is a minimal parabolic subgroup of $G$.
Then the corresponding compact manifold $G/P$, which will be called the \textit{boundary} is diffeomorphic to a sphere.
In fact, since $G$ is of real rank one, its Weyl group $W(G, A)$ consists of exactly two elements.
Thus the Bruhat decomposition assures that the left action of $P$ on $G/P$ has exactly two orbits: One is $\{eP\}$ and the other is $PgP$ for some $g \in G$ such that $gNg^{-1} = N_-$, where $N_-$ is the opposite of $N$.
Now
\[
PgP = NgP = g(g^{-1}Ng)P = gN_-P.
\]
Since the product map $N_- \times P \to G$ is a diffeomorphism onto its image, the $N_-$-orbits of $P$ in $G/P$ is diffeomorphic to $N_-$.
Thus $G/P$ as a manifold is a disjoint union of a point  and a Euclidean space, which must be a sphere.

\subsubsection{Local structure of the left action on the boundary}\label{sss:loc_left_act}
To study local structure of the left action of $G$ on $G/P$ around the point $\o = eP$, we use the homomorphism $l_*:\g = \mathrm{Lie}(G) \to \mathfrak{X}(G/P)$ defined by 
\[
l_*(X)_{gP} = \frac{d}{dt}|_{t = 0} \exp(-tX)gP
\]
for $gP \in G/P$.
This homomorphism can be rephrased as follows.
There is a natural anti-isomorphism of  $\g$ onto the algebra of the right-invariant vector fields on $G$.
This induces an anti-homomorphism of $\g$ into the space of smooth vector fields $\mathfrak{X}(G/P)$ on $G/P$.
Multiplying by $-1$, we obtain the homomorphism $l_*:\g  \to \mathfrak{X}(G/P)$.

Using the embedding $i: N_- \to G/P$ defined by $i(g) = gP$ and the diffeomorphism $\exp: \n_- = \mathrm{Lie}(N_-) \to N_-$, we obtain a homomorphism $\lambda = \exp^*\circ i^*\circ l_*$ of $\g$ into $\mathfrak{X}(\n_-)$.
In the local coordinate system $\exp \circ i:\n_- \to G/P$ around $\o \in G/P$, the homomorphism $l_*: \g \to \mathfrak{X}(G/P)$ can be described in terms of notions introduced in Subsection \ref{subsect:vfvs}.

\begin{prop}\label{prop:std_fixed_pt}
Let $\lambda: \g \to \mathfrak{X}(\n_-)$ be  the homomorphism defined as above and $\P(\n_-) \subset \mathfrak{X}(\n_-)$ be the subalgebra of polynomial vector fields on $\n_-$.
\begin{enumerate}
\item $\lambda(\n_-) \subset \P(\n_-)$.
\item Let $E \in \mathfrak{a}$ be the vector characterized by $[E, X] = rX$ for all $X \in \g_r$.
Then $\lambda(E) \in \P(\n_-)$ is the linear vector field corresponding to 
\[
\ad(E)|_{\n_-} = -2\mathrm{Id}_{\g_{-2}} - \mathrm{Id}_{\g_{-1}} \in \mathfrak{gl}(\n_-).
\]
\item $\lambda(\g) \subset \P(\n_-)$.
\end{enumerate}
\end{prop}
\begin{proof}
{(i)}
By the construction, for $X \in \n_-$, $\lambda(X)$ is the pull-back of a right-invariant vector field on $N_-$ by $\exp: \n_- \to N_-$.
More explicitly, the tangent vector at $Y \in \n_-$ of $\lambda(X)$ is given by the  differential at $t = 0$ of the curve $\gamma(t)$ on $\n_-$ satisfying
\[
\exp(\gamma(t)) = \exp(-tX)\exp(Y).
\]
Since $N_-$ is nilpotent, the Baker-Campbell-Hausdorff formula assures that $\gamma(t)$ is a polynomial function of $tX$ and $Y$.
Thus the tangent vector at $Y$ is a polynomial function of $Y$.

{(ii)}
Observe that $i:N_- \to G/P$ is $A$-equivariant where the domain is equipped with the action by conjugation and the range with the left action.
In particular, $A$ acts on $N_-$ by automorphism of Lie group.
By construction of $\lambda:\g \to \mathfrak{X}(\n_-)$, the claim follows immediately.

{(iii)}
By {(ii)}, we can show that for any integer $m$, the subspace
\[
\{X \in \mathfrak{X}(\n_-) \mid [E, X] = mX\}
\]
is contained in $\P(\n_-)$.
See the proof of Proposition 7.12 in \cite{Okada}.
Since $\g = \bigoplus \g_r$, we see that $\lambda(\g) \subset \P(\n_-)$.
\end{proof}

We will use the following lemma which follows immediately from the construction of the map $\lambda$.
\begin{lem}\label{lem:center_nil}
Let $\mathfrak{z} = \mathfrak{z}_{\mathfrak{X}(\n_-)}(\lambda(\n_-))$ be the centralizer of $\lambda(\n_-)$ in $\mathfrak{X}(\n_-)$.
\begin{enumerate}
\item $\mathfrak{z}$ is isomorphic to $\n_-$ as a Lie algebra.
\item $\mathfrak{z}$ is isomorphic to $\n_-$ as a representation of $MA$.
\item $\mathfrak{z} \subset \P(\n_-)$.
\end{enumerate}
\end{lem}
\begin{proof}
Since $i^* \circ l_*(\n_-) \subset \mathfrak{X}(N_-)$ is the space of right-invariant vector fields on $N_-$, its centralizer is the space of left-invariant vector fields.
So $\mathfrak{z}$ is isomorphic to $\n_-$ as a Lie algebra.
Moreover, they are isomorphic as representations of $MA$.
In fact, they are isomorphic to the isotropic representation at $e \in N_-$.
In particular, by the same argument for (iii) of Proposition \ref{prop:std_fixed_pt}, we see that $\mathfrak{z}\subset \P(\n_-)$.
\end{proof}

\subsubsection{The standard subgroup}\label{sss:std_subgrp}
Let $G$ be the group of orientation-preserving isometries of a rank one symmetric space of non-compact type and $G = KAN$ an Iwasawa decomposition.
We define a subgroup $\Gamma$ of $G$ to be a \textit{standard subgroup} if it is generated by a non-trivial element $a \in A$ and a lattice $\Lambda \subset N$ of $N$ satisfying $a\Lambda a^{-1} \subset \Lambda$.
We will give an explicit presentation of $\Gamma$.

Recall that $\n = \g_1 \oplus \g_2$ for the restricted-root decomposition $\g = \bigoplus_r \g_r$.
In particular, $N$ is at most 2-step nilpotent.
Thus a lattice $\Lambda$ of $N$ has a set of generators $b_1, \dots, b_{m_1}, c_1, \dots, c_{m_2}$ such that
\begin{align*}
[b_i, c_j] &= e,\\
[b_i, b_j] &\in \langle c_1, \dots, c_{m_2} \rangle,\\
[c_i, c_j] &=e,
\end{align*}
where $m_i = \dim \g_i$.
For $a \in A$, $\Ad(a)$ on $\g$ is of the form $\Ad(a)|_{\g_r} = e^{rt}\mathrm{Id}_{\g_r}$ for some $t \in R$.
Thus, if a nontrivial element $a \in A$ satisfies $a\Lambda a^{-1} \subset \Lambda$, then there exists an integer $k \geq 2$ such that  $\Ad(a)|_{\g_r} = k^r\mathrm{Id}_{\g_r}$ so that $ab_ia^{-1} = b_i^k$, $ac_ia^{-1} = c_i^{k^2}$.

\section{Cohomology of Lie algebra}\label{sect:Lie_cohom}
Let $\g = \mathfrak{k} \oplus \mathfrak{a} \oplus \n$ be an Iwasawa decomposition of a real simple Lie algebra of real rank one.
In this section, we compute certain cohomology of $\n$.
Cohomology of $\n$ with the coefficient in a finite-dimensional $\g$-module can be computed by using Theorem \ref{thm:Kostant}.
We compute $H^1(\n, \g)^\mathfrak{a}$ in Subsection \ref{subsect:Lie_cohom_fin}.
As we will see in Section \ref{sect:loc_rigid_act}, vanishing of $H^1(\n, \g)^\mathfrak{a}$ is equivalent to local rigidity of our action.
The goal of Subsection \ref{subsect:Lie_cohom_inf} is Corollary \ref{cor_0cohom} and Corollary \ref{cor_1cohom}, which will be used in Section \ref{sect:loc_rigid_jet}.
To compute such cohomology of an infinite-dimenional $\g$-module, we use Theorem \ref{thm:Vogan}.

\subsection{Cohomology of finite-dimensional modules}\label{subsect:Lie_cohom_fin}
By using Theorem \ref{thm:Kostant} and the classification in Subsection \ref{subsect:rank_one_Lie}, we obtain the following.
To reduce the notation, $\g_\C$, $\n_\C$, $(\g_0)_\C$, and $\mathfrak{a}_\C$ in Subsection \ref{subsect:rank_one_Lie} will be denoted by $\g$, $\n$, $\l$, and $\mathfrak{a}$, respectively.
\begin{lem}\label{lem:H1nga}
Let $\g$ be the complexification of a real simple Lie algebra of real rank one.
Using the notation for $\Delta(\g, \h)$ as in  Subsection \ref{subsect:rank_one_Lie},
\[
H^1(\n, \g)^\mathfrak{a} = \begin{cases}
\mathrm{F}^\l_{2e_2} \oplus \mathrm{F}^\l_{-2e_2} & (\g = \mathfrak{so}(4, \C) )\\
\mathrm{F}^\l_{2e_2} & (\g = \mathfrak{so}(m, \C), \; m\geq 5 ) \\
\mathrm{F}^\l_{-e_1 + 2e_2 - e_n} \oplus \mathrm{F}^\l_{e_1-2e_{n-1} +e_n} & (\g = \mathfrak{sl}(n, \C), \; n\geq 3 )\\
0 & (\g = \mathfrak{sp}(2n, \C) \;(n\geq 3), \; \mathfrak{f}_4)
\end{cases}
\]
as an $\l$-module, where $\mathrm{F}^\l_\lambda$ denotes the finite-dimensional irreducible $\l$-module with the highest weight $\lambda$.
\end{lem}
\begin{proof}
By Theorem \ref{thm:Kostant},
\[
H^1(\n, \mathrm{F}^\g_\lambda) = \bigoplus _\mu \mathrm{F}^\l_\mu 
\]
as an $\l$-module, where the sum is taken over $\mu = r_\a(\lambda + \delta(\g)) - \delta(\g)$ for the reflection $r_\a$ in a simple root $\a$.
We will determine the $\mathfrak{a}$-invariant summands.

When $\g = \mathfrak{so}(4, \C)$, $\g = \mathrm{F}^\g_{e_1 + e_2} \oplus \mathrm{F}^\g_{e_1 - e_2}$ as a $\g$-module.
As $\mathfrak{a}^*$ is spanned by $e_1$, a weight vector is $\mathfrak{a}$-invariant if and only if the coefficient of $e_1$ of the weight is 0.
Since the simple roots are $e_1 \pm e_2$ and $\delta(\g) = 2e_1$, it is not difficult to see that
\[
H^1(\n, \mathrm{F}^\g_{e_1 \pm e_2})^\mathfrak{a} = \bigoplus _\mu \mathrm{F}^\l_{\pm2e_2}.
\]

When  $\g = \mathfrak{so}(2n, \C)$, $n \geq 3$, $\g = \mathrm{F}^\g_{e_1 + e_2}$ as a $\g$-module.
Using the facts that $\mathfrak{a}^*$ is spanned by $e_1$, the simple roots are
\[
e_i - e_{i+1}, (i = 1, \dots, n-1), \;e_{n-1} + e_n,
\]
and 
\[
\delta(\g) = \sum_{i =1} ^n (n - i)e_i,
\]
it is easy to see that a summand $\mathrm{F}^\l_\mu$, $\mu = r_{\a}(e_1 + e_2 + \delta(\g)) - \delta(\g)$ is $\mathfrak{a}$-invariant only if $\a = e_1 -e_2$.
Since
\[
r_{e_1 - e_2}(e_1 + e_2 + \delta(\g)) - \delta(\g)= 2e_2,
\]
the claim follows.

When  $\g = \mathfrak{so}(2n+1, \C)$, $n \geq 2$, $\g = \mathrm{F}^\g_{e_1 + e_2}$ as a $\g$-module.
Since $\mathfrak{a}^*$ is spanned by $e_1$, the simple roots are
\[
e_i - e_{i+1}, (i = 1, \dots, n-1),\; e_n,
\]
and
\[
\delta(\g) = \sum_{i =1} ^n \frac{2n+1 - 2i}{2}e_i,
\]
the claim follows from the same argument as the case $\g = \mathfrak{so}(2n, \C)$.

When $\g = \mathfrak{sl}(n, \C), \; n\geq 3$,  $\g = \mathrm{F}^\g_{e_1 - e_n}$ as a $\g$-module.
As $\mathfrak{a}^*$ is spanned by $e_1 - e_n$, a weight vector is $\mathfrak{a}$-invariant if and only if the coefficients of $e_1$ and $e_n$ of the weight is the same.
Since the simple roots are
\[
e_i - e_{i+1}, (i = 1, \dots, n-1)
\]
and 
\[
\delta(\g) = \sum_{i =1} ^n \frac{n +1 - 2i}{2}e_i,
\]
it is easy to see that a summand $\mathrm{F}^\l_\mu$, $\mu = r_{\a}(e_1 - e_n + \delta(\g)) - \delta(\g)$ is $\mathfrak{a}$-invariant only if $\a = e_1 -e_2, e_{n-1} - e_n$.
Since
\[
r_{e_1 - e_2}(e_1 - e_n + \delta(\g)) - \delta(\g)= -e_1 + 2e_2 - e_n
\]
and
\[
r_{e_{n-1} - e_n}(e_1 - e_n + \delta(\g)) - \delta(\g)= e_1 - 2e_{n-1} + e_n,
\]
the claim follows.

When $\g = \mathfrak{sp}(2n, \C), \; n\geq 3$,  $\g = \mathrm{F}^\g_{2e_1}$ as a $\g$-module.
As $\mathfrak{a}^*$ is spanned by $e_1 +e_2$, a weight vector is $\mathfrak{a}$-invariant if and only if the sum of the coefficients of $e_1$ and $e_2$ of the weight is 0.
Since the simple roots are
\[
e_i - e_{i+1}, (i = 1, \dots, n-1), 2e_n
\]
and 
\[
\delta(\g) = \sum_{i =1} ^n (n +1 - i)e_i,
\]
it is easy to see that a summand $\mathrm{F}^\l_\mu$, $\mu = r_{\a}(e_1 + e_2 + \delta(\g)) - \delta(\g)$ is $\mathfrak{a}$-invariant only if $\a = e_2 - e_3$.
Since
\[
r_{e_2 - e_2}(e_1 + e_2 + \delta(\g)) - \delta(\g)= 2e_1 - e_2+ e_3,
\]
there is no $\mathfrak{a}$-invariant summands and the claim follows.

When $\g = \mathfrak{f}_4$,  $\g = \mathrm{F}^\g_{2e_1+ 2e_2}$ as a $\g$-module.
As $\mathfrak{a}^*$ is spanned by $e_1$, a weight vector is $\mathfrak{a}$-invariant if and only if the coefficient of $e_1$ of the weight is 0.
Since the simple roots are
\[
e_1 - e_2 -e_3 -e_4, 2e_2 -2e_3, 2e_3 - 2e_4, 2e_4,
\]
and 
\[
\delta(\g) = 11e_1 + 5e_2 + 3e_3 + e_4,
\]
it is easy to see that a summand $\mathrm{F}^\l_\mu$, $\mu = r_{\a}(2e_1 + 2e_2 + \delta(\g)) - \delta(\g)$ is $\mathfrak{a}$-invariant only if $\a = e_1 - e_2 -e_3 -e_4$.
Since
\[
r_{e_1 - e_2 -e_3 -e_4}(2e_1 + 2e_2 + \delta(\g)) - \delta(\g)= -e_1 - e_2+ e_3 + e_4,
\]
there is no $\mathfrak{a}$-invariant summands and the claim follows.

\end{proof}
\subsection{Cohomology of infinite-dimensional modules}\label{subsect:Lie_cohom_inf}
Let $\g$ be the complexification of a real simple Lie algebra of real rank one.
We will use the same notation for $\n$, $\l$, and $\mathfrak{a}$ as in Subsection \ref{subsect:Lie_cohom_fin}.
A weight $\lambda \in \h^*$ is \textit{orthogonal} to $\mathfrak{a}^*$ if $\langle \lambda, \mu \rangle = 0$ for all $\mu \in \mathfrak{a}^*$, where  $\langle \cdot, \cdot \rangle$ is the bilinear form on $\h^*$ induced by the Killing form on $\h$, and is $\l$-\textit{dominant} if $\langle \lambda, \mu \rangle > 0$ for all $\mu \in \Delta^+(\l, \h)$.
\begin{lem}\label{lem:cohominft}
Let $\lambda \in \h^*$ be the weight of an $\l$-lowest weight vector in $\n_-$ which is not the weight of a $\g$-lowest weight vector in $\g$.
Then for $w \in W(\g, \h)$, 
\[
\mu = w(\lambda  - \delta(\g)) - \delta(\g)
\]
is not orthogonal to $\mathfrak{a}^*$ if $\mu$ is $\l$-dominant.
\end{lem}
\begin{proof}
When $\g = \mathfrak{so}(n, \C)$, $n \geq 3$, an $\l$-lowest weight vector in $\n_-$ is a $\g$-lowest weight vector in $\g$, so there is no $\lambda$ satisfying the assumption.
The remaining cases are $\g = \mathfrak{sl}(n, \C)$ $(n \geq 3)$, $\mathfrak{sp}(2n, \C)$ $(n \geq 3)$, and $\mathfrak{f}_4$.

When $\g = \mathfrak{sl}(n, \C)$ $(n \geq 3)$, $\lambda = -e_1 +e_{n-1}$ or $-e_2 + e_n$.
We assume for simplicity $\lambda = -e_1 +e_{n-1}$.
Then
\[
\lambda - \delta(\g) = -\frac{n + 1}{2}e_1 - \frac{n-3}{2}e_2 - \dots + \frac{n-1}{2}e_{n-1} + \frac{n-1}{2}e_n.
\]
Since $\mathfrak{a}^*$ is spanned by $e_1 - e_2$, a weight $\mu$ is orthogonal to $\mathfrak{a}^*$ if and only if the coefficients of $e_1$ and $e_n$ are the same.
On the other hand, the difference between coefficients of $e_1$ and $e_n$ in $\delta(\g)$ is $n-1$.  
It follows that the coefficients of $e_1$ and $e_n$ in $\mu = w(\lambda - \delta(\g)) - \delta(\g)$ coincides for some $w \in W(\g, \h)$ only if the set of coefficients in $\lambda - \delta(\g)$ contains two elements which differ by $n-1$.
We see that there is no such two elements.
Thus $\mu$ is not orthogonal to $\mathfrak{a}$.
The case $\lambda = -e_2 + e_n$ follows by the same argument.

When $\mathfrak{sp}(2n, \C)$ $(n \geq 3)$, $\lambda = -e_1 -e_3$.
So
\[
\lambda - \delta(\g) = -(n+ 1)e_1 - (n-1)e_2 -(n-1)e_3 - (n-3)e_4 \dots - e_n.
\]
Since $\mathfrak{a}^*$ is spanned by $e_1 + e_2$, a weight $\mu$ is orthogonal to $\mathfrak{a}^*$ if and only if the sum of the coefficients of $e_1$ and $e_2$ is $0$.
Moreover, if such a weight $\mu$ is $\l$-dominant, the coefficient of $e_1$ is non-negative.
\footnote{
The assumption of $\l$-dominance is used only here.
}
Since the coefficient of $e_1$ in $\delta(\g)$ is $n$, if the coefficient of $e_1$ in $\mu = w(\lambda - \delta(\g)) - \delta(\g)$ is non-negative, it must be $1$.
But as the coefficient of $e_2$ in $\delta(\g)$ is $n-1$, the coefficient of $e_2$ in $\mu$ can not be $-1$.
Thus $\mu$ is not orthogonal to $\mathfrak{a}^*$.

When $\g = \mathfrak{f}_4$, $\lambda = -e_1-e_2-e_3-e_4$.
So
\[
\lambda -\delta(\g) = -12e_1 -6e_2 -4e_3 -2e_4.
\]
Since $\mathfrak{a}^*$ is spanned by $e_1$, a weight $\mu$ is orthogonal to $\mathfrak{a}^*$ if and only if the coefficient of $e_1$ is $0$.
Since the action of $W(\g, \h)$ preserves the bilinear form on $\h^*$, the set $\{\pm e_1 \pm e_2 \pm e_3 \pm e_4\} \cup \{\pm2e_i \mid 1\leq i \leq 4\}$ is invariant.
Observe that for each $\a$ in this set, $\langle \lambda -\delta(\g), \a \rangle$ is an integer multiple of $4|e_i|^2$.
Thus this is also true for $w(\lambda -\delta(\g))$.
In particular, the coefficient of $e_1$ in $w(\lambda -\delta(\g))$ is even.
Thus $w(\lambda -\delta(\g)) - \delta(\g)$ is not orthogonal to $\mathfrak{a}^*$.
\end{proof}

\begin{cor}\label{cor_cohominft}
Let $\lambda \in \h^*$ be the weight of an $\l$-lowest weight vector in $\n_-$ which is not the weight of a $\g$-lowest weight vector in $\g$.
If $V$ is an $\l$-finite $\g$-module with an infinitesimal character $\lambda - \delta(\g)$, then $H^*(\n, V)^\mathfrak{a} = 0$.
\end{cor}
\begin{proof}
By Theorem \ref{thm:Kostant}, the weight $\mu$ of an $\l$-highest weight vector in $H^*(\n, V)$ is of the form $\mu = w(\lambda  - \delta(\g)) - \delta(\g)$ for some $w \in W(\g, \h)$.
Observe that a weight vector is $\mathfrak{a}$-invariant if and only if its weight is othogonal to $\mathfrak{a}^*$ and that the weight of an $\l$-highest weight vector in a finite-dimensional $\l$-module is $\l$-dominant.
Thus the claim is immediate from Lemma \ref{lem:cohominft}.
\end{proof}
\begin{lem}\label{lem:0cohominft}
Let $\lambda \in \h^*$ be the weight of a $\g$-lowest weight vector in $\g$.
Assume a weight $\mu \in \h^*$ is $\l$-dominant, orthogonal to $\mathfrak{a}^*$, and
\[
\mu = w(\lambda  - \delta(\g)) - \delta(\g)
\]
for some $w \in W(\g, \h)$.
Then 
\[
\mu = 
\begin{cases}
2e_2 & (\g = \mathfrak{so}(n, \C),n\geq 4,\lambda = -e_1-e_2)\\
-2e_2 & (\g = \mathfrak{so}(4, \C),\lambda = -e_1+e_2) \\
-e_1+2e_2-e_n, e_1 - 2e_{n-1} +e_n&(\g = \mathfrak{sl}(n, \C),n\geq 3)\\
2e_1-2e_2+e_3+e_4&(\g = \mathfrak{sp}(2n, \C),n\geq3)\\
4e_2 + e_3 + e_4&(\g = \mathfrak{f}_4)
\end{cases}
\]
\end{lem}
\begin{proof}
When $\g = \mathfrak{so}(4, \C)$, $\lambda = -e_1\pm e_2$.
Let us first consider the case $\lambda = -e_1 - e_2$.
Then $\lambda - \delta(\g) = -2e_1 - e_2$.
So $\mu =  w(\lambda  - \delta(\g)) - \delta(\g)$ is orthogonal to $\mathfrak{a}$ only if $w(\lambda  - \delta(\g)) = e_1 +2e_2$ and $\mu = 2e_2$.
When $\lambda = -e_1 -e_2$, the claim follows from the above argument with $e_2$ replaced by $-e_2$.

When $\g = \mathfrak{so}(2n+1, \C)$ $(n \geq 2)$, $\lambda = -e_1 - e_2$.
So
\[
\lambda  - \delta(\g) = -\frac{2n+1}{2}e_1 -\frac{2n-1}{2}e_2 - \frac{2n-5}{2}e_3 -\dots - \frac{1}{2}e_n.
\]
Since the coefficient of $e_1$ in $\delta(\g)$ is $\frac{2n-1}{2}$,  $\mu$ is orthogonal to $\mathfrak{a}$ if and only if $w(e_2) = -e_1$.
It is not difficult to check that $\mu$ is $\l$-dominant only if 
\[
w(\lambda  - \delta(\g)) = \frac{2n-1}{2}e_1 +\frac{2n+1}{2}e_2 + \frac{2n-5}{2}e_3 +\dots + \frac{1}{2}e_n
\]
and $\mu = 2e_2$.

The case $\g = \mathfrak{so}(2n, \C)$ $(n \geq 3)$ is similar as above.
In this case, $\lambda = -e_1 - e_2$.
Comparing the coefficients of $\lambda  - \delta(\g)$ and $\delta(\g)$, we can show that $\mu$ is orthogonal to $\mathfrak{a}$ if and only if $w(e_2) = -e_1$.
By the $\l$-dominance, $\mu = 2e_2$.

When $\g = \mathfrak{sl}(n, \C)$ $(n \geq 3)$, $\lambda = -e_1 + e_n$ and
\[
\lambda - \delta(\g) = -\frac{n+1}{2}e_1 - \frac{n-3}{2}e_2 - \dots +\frac{n-3}{2}e_{n-1} + \frac{n+1}{2}e_n.
\]
As the difference between coefficients of $e_1$ and $e_n$ in $\delta(\g)$ is $n-1$, we see that $\mu$ is orthogonal to $\mathfrak{a}^*$ if and only if $w(e_1) = e_n$, $w(e_{n-1}) = e_1$ or $w(e_2) = e_n$, $w(e_n) = e_1$.
When  $w(e_1) = e_n$ and $w(e_{n-1}) = e_1$, $\mu$ is $\l$-dominant only if
\[
w(\lambda - \delta(\g)) = \frac{n-3}{2}e_1  + \frac{n+1}{2}e_2 + \frac{n-5}{2}e_3 + \dots -\frac{n-3}{2}e_{n-1} -\frac{n+1}{2}e_n
\]
and $\mu = -e_1 +2e_2 -e_n$.
When $w(e_2) = e_n$, $w(e_n) = e_1$, $\mu$ is $\l$-dominant only if
\[
w(\lambda - \delta(\g)) = \frac{n+1}{2}e_1  + \frac{n-3}{2}e_2 +\dots - \frac{n-5}{2}e_{n-2}-\frac{n-3}{2}e_{n-1}-\frac{n-3}{2}e_n
\]
and $\mu = e_1 -2e_{n-1} +e_n$.

When $\g = \mathfrak{sp}(2n, \C)$ $(n \geq 3)$, $\lambda = -2e_1$ and
\[
\lambda - \delta(\g) = -(n+2)e_1 - (n-1)e_2 - \dots -e_n.
\]
As the sum of the coefficients of $e_1$ and $e_2$ in $\delta(\g)$ is $2n-1$, we see that $\mu$ is orthogonal to $\mathfrak{a}^*$ if and only if $w$ maps $\{e_1, e_4\}$ onto $\{-e_1, -e_2\}$.
When $\mu$ is $\l$-dominant, the coefficient of $e_1$ in $\mu$ is non-negative.
Thus $w(e_1) = -e_1$, $w(e_4) = -e_2$.
Then $\mu$ is $\l$-dominant only if
\[
w(\lambda - \delta(\g)) = (n+2)e_1  + (n-3)e_2 + (n-1)e_3 + (n-2)e_4 + (n-4)e_5 +\dots +e_n
\]
and $\mu = 2e_1 -2e_2 +e_3 + e_4$.

When $\g = \mathfrak{f}_4$, $\lambda = -2e_1 -2e_2$ and
\[
\lambda - \delta(\g) = -13e_1 -7e_2 -3e_3 -e_4.
\]
Assume $\mu$ is orthogonal to $\mathfrak{a}^*$.
Then the coefficient of $e_1$ in $w(\lambda - \delta(\g))$ is $11$.
Let $\{c_1, c_2, c_3, c_4\}$ be the set of coefficients of $w(\lambda - \delta(\g))$ with $|c_i| \geq |c_{i+1}|$.
Since $W(\g, \h)$ preserves the bilinear form on $\h^*$, $13^2 + 7^2 + 3^1 + 1^2 = \sum_i c_i^2$.
It follows that $c_1 = 11$.
Using the fact that $\{\pm e_1 \pm e_2 \pm e_3 \pm e_4\} \cup \{\pm2e_i \mid 1\leq i \leq 4\}$ is $W(\g, \h)$-invariant and that $\langle \lambda - \delta(\g), \a \rangle$ is an integer multiple of $2|e_i|^2$ for each element $\a$ of this set, we see that the coefficients $c_i$ of $w(\lambda - \delta(\g))$ are integers.
Since $\{\pm 2e_i \pm 2e_j \mid 1 \leq i < j \leq 4\}$ is $W(\g, \h)$-invariant and the maximal value of $\langle \lambda - \delta(\g), \a\rangle$ for $\a \in \{\pm 2e_i \pm 2e_j \mid 1 \leq i < j \leq 4\}$ is $2(13 + 7) $, we see $|c_1| + |c_2| = 13 + 7$.
Thus $|c_2| = 9$.
By the equation $13^2 + 7^2 + 3^1 + 1^2 = \sum_i c_i^2$, we obtain $|c_3| = 4$ and $|c_4| = 2$.
Now it is easy to check that $\mu$ is $\l$-dominant only if
\[
w(\lambda - \delta(\g)) = 11e_1 + 9e_2 + 4e_3 + 2e_4
\]
and $\mu = 4e_2 + e_3 + e_4$.
\end{proof}

Thus when $V$ is a $\g$-module with the same infinitesimal character as that of $\g$, unlike Corollary \ref{cor_cohominft}, $H^*(\n, V)^\mathfrak{a}$ does not necessarily vanish.
In fact, when $V = \g$, as we have seen in Subsection \ref{subsect:Lie_cohom_fin}, $H^1(\n, \g)^\mathfrak{a} \neq 0$ if $\g = \mathfrak{so}(n, \C)$ or $\mathfrak{sl}(n, \C)$.

We define a $\g$-module $V$ to be \textit{$\mathfrak{a}$-bounded below} if for all positive weights $\mu \in \mathfrak{a}^*$, the set $\{ \langle \mu, \lambda \rangle \mid \lambda \in \mathrm{wt}(V)\} \subset \R$ is bounded below, where $\mathrm{wt}(V)$ denotes the set of weights in $V$.
\begin{cor}\label{cor_0cohominft}
Let $\lambda \in \h^*$ be the weight of a $\g$-lowest weight vector in $\g$.
If $V$ is an $\l$-finite $\g$-module with an infinitesimal character $\lambda-\delta(\g)$ which is $\mathfrak{a}$-bounded below, then $H^0(\n, V)^\mathfrak{a}= 0$.
\end{cor}
\begin{proof}
By Theorem \ref{thm:Vogan}, the weight $\mu$ of an $\l$-highest weight vector in $H^0(\n, V)^\mathfrak{a}$ must be $\mu$ as in Lemma \ref{lem:0cohominft}.
On the other hand, an $\l$-highest weight vector in $H^0(\n, V)^\mathfrak{a}$ is a $\g$-highest weight vector in $V$.
Since $V$ is $\mathfrak{a}$-bounded below, a $\g$-highest weight vetor in $V$ must be a highest weight vector of a finite-dimensional $\g$-submodule.
Thus its weight $\mu$ must be $\g$-dominant.
But weights $\mu$ in in Lemma \ref{lem:0cohominft} are not $\g$-dominant.
Thus $H^0(\n, V)^\mathfrak{a} = 0$.
\end{proof}

We define a weight $\lambda \in \h^*$ to be $\mathfrak{a}^*$-\textit{nonnegative} if $\langle \mu, \lambda \rangle \leq 0$ for all positive weights $\mu \in \mathfrak{a}^*$.

\begin{prop}\label{prop:1cohominft_1}
Assume $\g = \mathfrak{so}(n, \C)$, $n \geq 4$ or $\mathfrak{sl}(n, \C)$, $n \geq 3$.
Let $\lambda \in \h^*$ be the weight of a $\g$-lowest weight vector in $\g$, and $V$ an $\l$-finite $\g$-module with an infinitesimal character $\lambda - \delta(\g)$.
Assume the weights of $V$ are $\mathfrak{a}^*$-nonnegative.
Then $H^1(\n, V)^\mathfrak{a}= 0$.
\end{prop}
Assume $H^1(\n, V)^\mathfrak{a} \neq 0$.
We will again use the explicit description of the root system $\Delta(\g, \h)$ as in Subsection \ref{subsect:rank_one_Lie} to obtain a contradiction.

\begin{proof}[Proof in the case $\g = \mathfrak{so}(4, \C)$]
The weight $\lambda$ of a $\g$-lowest weight vector in $\g$ is $\lambda = -e_1\pm e_2$.
When $\lambda = -e_1 - e_2$, by Lemma \ref{lem:0cohominft}, $H^1(\n, V)$ has an $\l$-highest weight vector of weight $2e_2$.
Let $f:\n \to V$ be a non-zero cocycle of weight $2e_2$.
Then $f(\g_{e_1+e_2})$ or $f(\g_{e_1-e_2})$ is non-zero.
So $V$ contains a weight vector of weight $e_1+3e_2$ or $e_1+e_2$.

We will show $V$ does not contain a weight vector of weight $e_1+3e_2$ or $e_1+e_2$.
Since $V$ is an $\l$-finite $\g$-module with $\mathfrak{a}^*$-nonnegative weights, a vector in $V$ generates a $\g_-$-submodule which contains a $\g$-lowest weight.
On the other hand, the weight of a $\g$-lowest weight vector in $V$ is of the form $w(\lambda-\delta(\g)) + \delta(\g)$.
The weights appears in the $\g_-$-submodule of $V$ generated by a weight vector of weight $e_1+3e_2$ are $e_1+3e_2, 4e_2, 2e_2$, while that of weight $e_1+e_2$ are $e_1+e_2, 4e_2, 2e_2$.
It is easy to see that none of them are of the form  $w(\lambda-\delta(\g)) + \delta(\g)$.

When $\lambda = -e_1 + e_2$, the claim follows from the above argument with $e_2$ replaced by $-e_2$.
\end{proof}

\begin{proof}[Proof in the case $\g = \mathfrak{so}(m, \C), m \geq 5$]
By Lemma \ref{lem:0cohominft}, $H^1(\n, V)$ has an $\l$-highest weight vector of weight $\mu = 2e_2$.
Let $f:\n \to V$ be a non-trivial $\l$-highest cocycle of weight $\mu$.
Since $\g_{e_1-e_2}$ generates $\n$ as an $\l_+$-module, an $\l$-highest cocycle $f$ is determined by $f|_{\g_{e_1-e_2}}$.

Let $\g' = \h \oplus \bigoplus_{\a \in \Delta(\g', \h)} \g_\a$ be the subalgebra of $\g$ where 
\[
\Delta(\g', \h) = \{ \pm e_1 \pm e_2\}.
\]
Put $\n' = \n \cap \g'$, and $\l' = \l \cap \g'$.
Now we will show that the restriction of $f$ to $\n'$ gives a non-zero $\l'$-highest weight vector in $H^1(\n', V')$ of weight $\mu$, where $V'$ denotes the $\g'$-subalgebra of $V$ generated by $f(\g_{e_1 - e_2})$.
It suffices to show that $f|_{\n'}$ is not a boundary.
If $f|_{\n'}$ is a boundary, there exists $v \in V$ with $f|_{\n'} = dv$.
Then $v$ is a weight vector of weight $\mu$.
Such a weight vector is $\l$-highest: as $\mu$ is orthogonal to $\mathfrak{a}$, the $\l$-submodule generated by $v$ admits an $\l$-highest weight vector of weight orthogonal to $\mathfrak{a}$.
By Lemma \ref{lem:0cohominft}, the weight of the $\l$-highest weight vector is $\mu$.
Thus $v$ is $\l$-highest.
Since $v$ is $\l$-highest, $f-dv$ is an $\l$-highest cocycle.
By $(f-dv)(\g_{e_1-e_2}) = 0$, we obtain $f - dv =0$.
So $f$ is a boundary, which contradicts to the assumption.

Replacing $f$ if necessary, we may assume $V'$ admits a $\g'$-infinitesimal character.
By Theorem \ref{thm:Vogan}, the infinitesimal character must be $\mu + \delta(\g') = e_1 + 2e_2$.
By the same argument as in the proof of the case $\g = \mathfrak{so}(4, \C)$, $V'$ does not contain a weight vector of weight $e_1 + e_2$.
This contradicts to the fact that $f(\g_{e_1-e_2})$ is of weight $e_1 -e_2 + \mu = e_1 +e_2$.
\end{proof}

\begin{proof}[Proof in the case $\g = \mathfrak{sl}(n, \C), n\geq 3$]
By Lemma \ref{lem:0cohominft}, $H^1(\n, V)$ has an $\l$-highest weight vector of weight $\mu = -e_1+2e_2-e_n, e_1 - 2e_{n-1} +e_n$.
Let us first consider the case $\mu = -e_1+2e_2-e_n$.
Let $f:\n \to V$ be a non-trivial $\l$-highest cocycle of weight $\mu$.

We will show $f = 0$ if $f|_{\g_{e_1-e_2}} = 0$.
Assume $f(\g_{e_1-e_2}) =0$.
By $[\g_{e_1-e_2}, \g_{e_1-e_n}] = 0$ and the cocycle equation, $f(\g_{e_1-e_n})$ is annihilated by $\g_{e_1-e_2}$.
If $f(\g_{e_1-e_n}) \neq 0$, this is a weight vector of weight $\mu + e_1-e_n = 2e_2-2e_n$.
Put $\mathfrak{s}_{e_1-e_2} = \h \oplus \g_{e_1-e_2} \oplus \g_{-e_1+e_2}$
As the weights of $V$ are $\mathfrak{a}$-nonnegative, the $\mathfrak{s}_{e_1-e_2}$-module generated by $f(\g_{e_1-e_n})$ is finite dimensional with a $\mathfrak{s}_{e_1-e_2}$-highest weight vector in $f(\g_{e_1-e_n})$.
But its weight $2e_2-2e_n$ is not $\mathfrak{s}_{e_1-e_2}$-dominant, which is a contradiction.
Thus $f(\g_{e_1-e_n}) = 0$.
Then by $[\g_{e_1-e_2}, \g_{e_{n-1}-e_n}] \subset \g_{e_1-e_n}$ and the cocycle equation, $f(\g_{e_{n-1}-e_n})$ is annihilated by $\g_{e_1-e_2}$.
Since the weight $\mu + e_{n-1}-e_n = -e_1+2e_2+e_{n-1}-2e_n$ is not $\mathfrak{s}_{e_1-e_2}$-dominant, by the same argument as above, we obtain $f(\g_{e_{n-1}-e_n}) = 0$.
As $\g_{e_1-e_2}$, $\g_{e_{n-1}-e_n}$, and $\g_{e_1-e_n}$ generates $\n$ as an $\l_+$-module, we see $f = 0$.

Let $\g' = \h \oplus \bigoplus_{\a \in \Delta(\g', \h)} \g_\a$ be the subalgebra of $\g$ where 
\[
\Delta(\g', \h) = \{ \pm (e_i  -e_j) \mid 1 \leq i < j \leq n-1\}
\]
and put $\n' = \n \cap \g'$, and $\l' = \l \cap \g' = \l$.
By the same argument as in the case of $\g = \mathfrak{so}(m, \C)$, we obtain an $\l'$-highest weight vector in $H^1(\n', V')$ of weight $\mu$, where $V'$ denotes the $\g'$-subalgebra of $V$ generated by $f(\g_{e_1 - e_2})$.

Replacing $f$ if necessary, we may assume $V'$ admits a $\g'$-infinitesimal character.
Observe that
\[
\delta(\g' ) = \frac{n-2}{2}e_1 + \frac{n-4}{2}e_2 + \dots -\frac{n-4}{2}e_{n-2} - \frac{n-2}{2}e_{n -1}.
\]
By Theorem \ref{thm:Vogan}, the infinitesimal character must be
\[
\mu + \delta(\g') = \frac{n-4}{2}e_1 + \frac{n}{2}e_2 + \dots - \frac{n-2}{2}e_{n -1} - e_n.
\]
Thus the weight of a $\g'$-lowest weight vector is of the form $w(\mu + \delta(\g')) + \delta(\g')$ for some $w \in W(\g', \h)$.
We see that the weight of this form appears in the $\g'_-$-submodule generated by $f(\g_{e_1 - e_2})$ only if $w(\mu + \delta(\g')) + \delta(\g') = e_{n-1} - e_n$.
But $V'$ does not contain a $\g'$-lowest weight vector of weight $e_{n-1} - e_n$.
In fact, if there is such a weight vector, $V$ also has a $\g'$-lowest weight vector of weight $e_{n-1} - e_n$.
Considering the $\g$-infinitesimal character of $V$, we see that the weight vector is not $\g$-lowest.
So it is not annihilated by $\g_{-e_{n-1} +e_n}$.
Then applying $\g_{-e_{n-1} +e_n}$, we obtain a $\g'$-lowest weight vector in $V$ of weight $e_{n-1} - e_n -e_{n-1} +e_n = 0$.
Considering the $\g$-infinitesimal character of $V$ again, this is a contradiction.
\end{proof}
\begin{prop}\label{prop:1cohominft_2}
Assume $\g = \mathfrak{sp}(2n, \C)$, $n \geq 3$ or $\mathfrak{f}_4$.
Let $\lambda \in \h^*$ be the weight of a $\g$-lowest weight vector in $\g$, and $V$ an $\l$-finite $\g$-module with an infinitesimal character $\lambda - \delta(\g)$ which is $\mathfrak{a}$-bounded below.
Then $H^1(\n, V)^\mathfrak{a}= 0$.
\end{prop}
\begin{proof}[Proof in the case $\g = \mathfrak{sp}(2n, \C), n\geq 3$]
By Lemma \ref{lem:0cohominft}, $H^1(\n, V)$ has an $\l$-highest weight vector of weight $\mu =2e_1-2e_2+e_3+e_4$.
Let $f:\n \to V$ be a non-trivial $\l$-highest cocycle of weight $\mu$.
Since $\n$ is generated by the $\l_+$-submodule generated by $\g_{e_2 - e_3}$ as Lie algebra, $f = 0$ if $f|_{\g_{e_2-e_3}} = 0$.

Let $\g' = \h \oplus \bigoplus_{\a \in \Delta(\g', \h)} \g_\a$ be the subalgebra of $\g$ where 
\[
\Delta(\g', \h) = \{ \pm e_2 \pm e_3 , \pm 2e_2, \pm 2e_3\}
\]
and put $\n' = \n \cap \g'$, and $\l' = \l \cap \g'$.
By the same argument as in the case of $\g = \mathfrak{so}(m, \C)$, we obtain an $\l'$-highest weight vector in $H^1(\n', V')$ of weight $\mu$, where $V'$ denotes the $\g'$-subalgebra of $V$ generated by $f(\g_{e_2 - e_3})$.

Replacing $f$ if necessary, we may assume $V'$ admits a $\g'$-infinitesimal character.
By Theorem \ref{thm:Vogan}, the infinitesimal character must be
\[
\mu + \delta(\g') = (2e_1-2e_2+e_3+e_4) + (2e_2 + e_3) = 2e_1 + 2e_3+e_4.
\]
Thus the weight of $\g'$-lowest weight vector is of the form $w(\mu + \delta(\g')) + \delta(\g')$ for some $w \in W(\g', \h)$.
Observe that the coefficient of $e_2$ in $w(\mu + \delta(\g')) + \delta(\g')$ is non-negative.
On the other hand, the weights in the $\g'_-$-module generated by $f(\g_{e_2 - e_3})$, the weight of which is $e_2 - e_3 + \mu = 2e_1-e_2+e_4$, have negative coefficients for $e_2$.
This is a contradiction.
\end{proof}

\begin{proof}[Proof in the case $\g = \mathfrak{f}_4$]
By Lemma \ref{lem:0cohominft}, $H^1(\n, V)$ has an $\l$-highest weight vector of weight $\mu =4e_2+e_3+e_4$.
Since $\n$ is generated by the $\l_+$-submodule generated by $\g_{e_1 -e_2 - e_3 -e_4}$ as Lie algebra, $f = 0$ if $f|_{\g_{e_1 -e_2 - e_3 -e_4}} = 0$.

Let $\g' = \h \oplus \bigoplus_{\a \in \Delta(\g', \h)} \g_\a$ be the subalgebra of $\g$ where
\[
\Delta(\g', \h) = \{ \pm (e_1 -e_2 - e_3) \pm e_4\} \cup \{\pm2e_4\}
\]
and put $\n' = \n \cap \g'$, and $\l' = \l \cap \g'$.
Observe that $\g'$ is a reductive Lie algebra with its semisimple part isomorphic to $\mathfrak{sl}(3, \C)$.
By the same argument as in the case of $\g = \mathfrak{so}(m, \C)$, we obtain an $\l'$-highest weight vector in $H^1(\n', V')$ of weight $\mu$, where $V'$ denotes the $\g'$-subalgebra of $V$ generated by $f(\g_{e_1 -e_2 - e_3 -e_4})$.

Replacing $f$ if necessary, we may assume $V'$ admits a $\g'$-infinitesimal character.
By Theorem \ref{thm:Vogan}, the infinitesimal character must be
\[
\mu + \delta(\g') = (4e_2+e_3+e_4) + (e_1 -e_2-e_3+ e_4 ) = e_1 +3e_2+ 2e_4.
\]
Thus the weight $\nu$ of a $\g'$-lowest weight vector satisfies $|\mu + \delta(\g')| =|\nu- \delta(\g')|$.
On the other hand, the weights in the $\g'_-$-module generated by $f(\g_{e_1 -e_2 - e_3 -e_4})$, the weight of which is $e_1 -e_2 - e_3 -e_4 + \mu = e_1 + 3e_2$, are of the form $(1-k)e_1 +(3 +k)e_2 + ke_3 +le_4$ for a non-negative integer $k$ and an integer $l$.
So
\[
\nu - \delta(\g') =  -ke_1 +(k + 4)e_2 + (k + 1) e_3 +le_4
\]
for a non-negative integer $k$ and an integer $l$.
Thus $|\mu + \delta(\g')| < |\nu- \delta(\g')|$, which is a contradiction.
\end{proof}
Let $G$ be the group of orientation-preserving isometries of a rank one symmetric space of non-compact type,  $G = KAN$ an Iwasawa decomposition, and $M$ the centralizer of $A$ in $K$.
Then $P = MAN$  is a minimal parabolic subgroup.
Recall that we defined the homomorphism $\lambda:\g \to \P(\n_-)$ in Subsection \ref{subsect:std_boundary}.
Consider the representation of $\g$ on $\P(\n_-)$ via $\lambda:\g \to \P(\n_-)$.
Let $\l \subset \g$ be the subalgebra corresponding to $MA \subset G$.

\begin{lem}\label{lem_pinft}
The $\g_\C$-module $V = \mathrm{Poly}(\n_-)_\C$ admits a decomposition $V = \bigoplus V_{\a_i}$ into a finite sum of $\g_\C$-submodules, where the sum is taken over the set $\{\a_i\}_i$ of weights of $\l_\C$-lowest weight vectors in $(\n_-)_\C$ and $V_{\a_i}$ is a $\g_\C$-submodule with an infinitesimal character $\a_i -\delta(\g_\C)$.
\end{lem}
\begin{proof}
It suffices to show that the weight of a weight vector in $V = \mathrm{Poly}(\n_-)_\C$ annihilated by the sum $(\g_\C)_-$ of negative root spaces is $\a_i$.
Since $(\g_\C)_- = (\n_-)_\C \oplus (\l_\C)_-$, $V^{(\g_\C)_-} = V^{(\n_-)_\C} \cap V^{ (\l_\C)_-}$.
By Lemma \ref{lem:center_nil}, the centralizer of $\lambda(\n_-)$ in $\P(\n_-)$ is an $\l$-submodule which is isomorphic to $\n_-$.
Thus $V^{(\n_-)_\C}$ is isomorphic to $(\n_-)_\C$ as an $\l_\C$-module.
Thus the weight of an $\l_\C$-lowest wight vector is $\a_i$.
\end{proof}

As we mentioned in Subsection \ref{subsect:vfvs}, since $\n_-$ is an $\l$-module,
\[
\P(\n_-) = S(\n_-^*) \otimes \n_-
\]
as an $\l$-module.
Thus $\P(\n_-)_\C$ is $\l_\C$-finite as a $\g_\C$-module.
\begin{cor}\label{cor_0cohom}
$H^0(\n, \P(\n_-))^\mathfrak{a} = 0$.
\end{cor}
\begin{proof}
We will show the complexification $H^0(\n_\C, \P(\n_-)_\C)^{\mathfrak{a}_\C}$ is vanished.
By Lemma \ref{lem_pinft}, it suffices to show $H^0(\n_\C, V_{\a_i})^{\mathfrak{a}_\C} = 0$ for all $\a_i$.
Since $V_{\a_i}$ admits an infinitesimal character $\alpha_i - \delta(\g)$, this is immediate from Corollary \ref{cor_cohominft} and Corollary \ref{cor_0cohominft}.
\end{proof}
\begin{cor}\label{cor_1cohom}
The map
\[
H^1(\n, \g)^\mathfrak{a} \to H^1(\n, \P(\n_-))^\mathfrak{a}
\]
induced by $\lambda:\g \to \P(\n_-)$ is an isomorphism.
\end{cor}
\begin{proof}
We will show the complexification $H^1(\n_\C, \g_\C)^{\mathfrak{a}_\C} \to H^1(\n_\C, \P(\n_-)_\C)^{\mathfrak{a}_\C}$ is an isomorphism.
Since $\lambda:\g \to \P(\n_-)$ is injective, we have a short exact sequence
\[
0 \to \g \to \P(\n_-) \to \P(\n_-)/\lambda(\g) \to 0.
\]
Thus it suffices to show that $H^i(\n_\C, \P(\n_-)_\C/\lambda(\g)_\C)^{\mathfrak{a}_\C} = 0$ for $i = 0, 1$.
By Lemma \ref{lem_pinft}, $V = \mathrm{Poly}(\n_-)_\C$ admits a decomposition $V = \bigoplus V_{\a_i}$ into $\g_\C$-submodules $\bigoplus V_{\a_i}$ with an infinitesimal character $\a_i -\delta(\g_\C)$.
So its quotient $V' = \P(\n_-)_\C/\lambda(\g)_\C$ also admits a  decomposition $V' = \bigoplus V'_{\a_i}$, where $\bigoplus V'_{\a_i}$ has an infinitesimal character $\a_i -\delta(\g_\C)$.
Let $\a_0$ be the weight of a $\g_\C$-lowest weight vector in $\g_\C$.
By Corollary \ref{cor_cohominft},  $H^*(\n_\C, V'_{\a_i})^{\mathfrak{a}_\C} = 0$ for $\a_i \neq \a_0$.
Moreover, since $\P(\n_-) = S(\n_-^*) \otimes \n_-$ as $\l$-modules, $\P(\n_-)$ is $\mathfrak{a}$-bounded below.
By Corollary \ref{cor_0cohominft}, $H^0(\n_\C, V'_{\a_0})^{\mathfrak{a}_\C} = 0$.
Thus it remains to show $H^1(\n_\C, V'_{\a_0})^{\mathfrak{a}_\C} = 0$.

When $\g_\C = \mathfrak{sp}(2n, \C)$, $n \geq 3$ or $\mathfrak{f}_4$, by Proposition \ref{prop:1cohominft_2}, $H^1(\n_\C, V'_{\a_0})^{\mathfrak{a}_\C} = 0$.
So we may assume $\g_\C = \mathfrak{so}(n, \C)$, $n \geq 4$ or $\mathfrak{sl}(n, \C)$, $n \geq 3$.
By Proposition \ref{prop:1cohominft_1},  it suffices to show that the weights of $V'_{\a_0}$  are $\mathfrak{a}_\C^*$-nonnegative.

Under the isomorphism $\P(\n_-) = S(\n_-^*) \otimes \n_-$ as $\l$-modules, the subspace spanned by weight vectors in $\P(\n_-)$ of $\mathfrak{a}^*$-negative weights is
\[
\g_{-2} \oplus \g_{-1} \oplus (\g_{-1}^* \otimes \g_{-2}) \subset S(\n_-^*) \otimes \n_-.\]
On the other hand, the weights in $\lambda(\n_-)$ are $\mathfrak{a}^*$-negative and by Lemma \ref{lem:center_nil}, the weights in $Z(\lambda(\n_-)) = Z_{\P(\n_-)}(\lambda(\n_-))$ are also $\mathfrak{a}^*$-negative.

When $\g_\C = \mathfrak{so}(n, \C)$, we may assume $\g_{-2} = 0$.
Then $\lambda(\n_-) = Z(\lambda(\n_-))$ is of dimension equal to $\n_- = \g_{-1}$.
Thus $\P(\n_-)/\lambda(\g)$ does not have $\mathfrak{a}^*$-negative weights.

When $\mathfrak{sl}(n, \C)$, $n \geq 3$, $\g_{-2}$ is of dimension one.
Since $Z_{\n_-}(\n_-) = \g_{-2}$, $\lambda(\n_-)$ and $Z(\lambda(\n_-))$ span a subspace of dimension $2 \dim\g_{-1} + \dim\g_{-2}$, which is euqal to the dimension of $\g_{-2} \oplus \g_{-1} \oplus (\g_{-1}^* \otimes \g_{-2})$.
Thus weight vectors in $\P(\n_-)$ of $\mathfrak{a}^*$-negative weights are contained in $\lambda(\n_-) + Z(\lambda(\n_-))$.
So $V'_{\a_0} = V_{\a_0}/\lambda(\g_\C)$ does not have $\mathfrak{a}_\C^*$-negative weights.
\end{proof}

Using Corollary \ref{cor_0cohom}, we can show the following lemma which will be used in the proof of Proposition \ref{prop:nonloc_rigid}.
\begin{lem}\label{lem:nonlocrigid}
Assume $\g = \mathfrak{su}(n, 1)$, $n \geq 3$.
Let $N_{\P(\n_-)}(\lambda(\n))$ be the normalizer of $\lambda(\n)$ in $\P(\n_-)$ and $Z_{\P(\n_-)}(\lambda(\mathfrak{a})) = \P(\n_-)^\mathfrak{a}$ the centralizer of $\lambda(\mathfrak{a})$ in $\P(\n_-)$.
Then $N_{\P(\n_-)}(\lambda(\n)) \cap Z_{\P(\n_-)}(\lambda(\mathfrak{a}))  = \lambda(\g^\mathfrak{a})$.
\end{lem}
\begin{proof}
Put $\mathfrak{q} = N_{\P(\n_-)}(\lambda(\n)) \cap Z_{\P(\n_-)}(\lambda(\mathfrak{a}))$.
We will first show $\mathfrak{q} \subset N_{\P(\n_-)}(\lambda(\n))$.
Fix $X \in \mathfrak{q}$.
Since $[\lambda(\g_{-2}), \P(\n_-)^\mathfrak{a}] \subset \lambda(\g_{-2})$, $X \in  N_{\lambda(\g_{-2})}(\P(\n_-))$.
Fix $Y \in \g_{-2}\setminus\{0\}$.
Then $[Y, \g_{1}] = \g_{-1}$.
Applying $\ad(\lambda(Y))^2$ to $[X, \lambda(\g_1)] \subset \lambda(\g_1)$, we see that $X \in  N_{\P(\n_-)}(\lambda(\n))$.

Identifying $\g$ with its image $\lambda(\g)$ by $\lambda$, we see that for $X \in \mathfrak{q}$, $\ad(X)|_{\lambda(\n)}$ defines an $\mathfrak{a}$-invariant cocycle on $\n$ with its value in $\g$.
By Corollary \ref{cor_0cohom}, the map $\mathfrak{q} \to Z^1(\n, \g)^\mathfrak{a}$ is injective, where $ Z^1(\n, \g)^\mathfrak{a}$ denotes the space of $\mathfrak{a}$-invariant cocycles.
This induces the injective map  $\mathfrak{q}/\lambda(\g^\mathfrak{a}) \to H^1(\n, \g)^\mathfrak{a}$.
We will show the complexification $\mathfrak{q}_\C/\lambda((\g_\C)^{\mathfrak{a}_\C})$ is vanished.
If $\mathfrak{q}_\C/\lambda((\g_\C)^{\mathfrak{a}_\C}) \neq 0$, by Lemma \ref{lem:H1nga}, the weight of an $\l_\C$-highest weght vector is $-e_1 +2e_2 -e_n$ or $e_1 -2e_{n-1} +e_n$.
Assume the weight is $-e_1 +2e_2 -e_n$.
Then there is a weight vector $X \in \mathfrak{q}_\C$ of weight $-e_1 +2e_2 -e_n$.
Since $\ad(X)|_\lambda(\n) \neq 0$, we see $[X, \lambda(\g_{e_1-e_2})] = \lambda(\g_{e_2 - e_n})$.
Applying $\ad(\lambda(\g_{-e_1+e_2}))$ to this equation and using $[X, \lambda(\g_{-e_1+e_2})] \in \lambda((\n_-)_\C)$, we obtain $X \in \g$ which is a contradiction.
The argument for $e_1 -2e_{n-1} +e_n$ is the same.
We proved $\mathfrak{q}_\C/\lambda((\g_\C)^{\mathfrak{a}_\C}) = 0$.
\end{proof}

\section{Cohomology of the standard subgroup}\label{section:vanish}
Let $G$ be a group of orientation-preserving isometries of a rank one symmetric space of non-compact type, and $\Gamma$ its standard subgroup.
The goal of this section is to prove \ref{prop:vanish}.
Recall that $\Gamma$ has a finite generating set $a, b_1, \dots, b_{m_1}, c_1, \dots, c_{m_2}$ as in \ref{sss:std_subgrp}.
\begin{lem}\label{lem:cohom_isotro}
Let $V$ be a vector space, and $\rho_s:\Gamma \to \mathrm{GL}(V)$ the representation defined by $\rho_s(\Lambda) = \{\id_V\}$ and $\rho_s(a) = s\id_V$ for a constant $s > 0$.
Then
\begin{enumerate}
\item $H^0(\Gamma, V) =  0$ if $s \neq 1$.
\item $H^1(\Gamma, V) =  0$ if $s \neq 1, k, k^2$.
\end{enumerate}
\end{lem}
\begin{proof}
{(i)}
Since $H^0(\Gamma, V)$ can be identified with the space $V^\Gamma$ of $\Gamma$-fixed vectors in $V$, the claim is immediate.

{(ii)}
Let $\beta_v(g) = v - \rho(g)v$ $(g \in \Gamma)$ be the coboundary given by $v \in V$.
Given a cocycle $\alpha:\Gamma \to V$,
Since $s \neq 1$, there is a unique $v \in V$ satisfying $\alpha(a) = \beta_v(a)$.
So we may assume $\alpha(a) = 0$.
For any $b_i \in \ b_1, \dots, b_{m_1} \}$, as $\a$ is a cocycle, 
\[
\alpha(ab_i) = s\alpha(b_i) + \alpha(a).
\]
On the other hand, using the relation $ab_i = b_i^{k}a$, 
\[
\alpha(ab_i) = \alpha(b_i^{k}a) =  \alpha(a) + \alpha(b_i^{k}) = \alpha(a) + k\alpha(b_i).
\]
Thus
\[
(s -k)\alpha(b_i) = 0.
\]
Since $s \neq k$, we see $\alpha(b_i) = 0$.
Similarly, for any $c_i \in \{ c_1, \dots, c_{m_1} \}$, using the relation $ac_i = c_i^{k^2}a$ and the assumption $s \neq k^2$, we obtain $\alpha(c_i) = 0$.
Thus the claim follows.
\end{proof}

Since $\Gamma \subset P \subset G$, the subalgebra $\p \subset \g $ is invariant under the adjoint representation of $\Gamma$ on $\g$.
The induced representation of $\Gamma$ on $\g/\p$ will also be called the adjoint representation.
\begin{prop}\label{prop:vanish}
Consider the adjoint representation of a standard subgroup $\Gamma $ on $\g/\p$.
Then $H^1(\Gamma, \g/\p) = 0$.
\end{prop}
\begin{proof}
Recall that $\g$ is graded $\g = \bigoplus_{i = -2}^2 \g_i$ so that $\p = \bigoplus_{i \geq 0} \g_i$.
Put $V = \g/\p$, and $W = (\bigoplus_{i \geq -1} \g_i)/\p$.
Then $V/W = \g/(\oplus_{i \geq -1} \g_i)$.
To prove $H^1(\Gamma, V) = 0$, it suffices to show $H^1(\Gamma, W) = 0$ and $H^1(\Gamma, V/W) = 0$.

Since $[\g_i, \g_j] \subset \g_{i+j}$, the adjoint representation of $\n = \g_1 \oplus \g_2$ on $V/W$ and $V$ are trivial.
Thus the representations of $\Lambda \subset N$ on $W$ and $V/W$ are also trivial.
Since $a$ acts on $W$ by $k^{-1}\id_W$ and on $V/W$ by $k^{-2}\id_{V/W}$, by Lemma \ref{lem:cohom_isotro}, $H^1(\Gamma, W) = H^1(\Gamma, V/W) = 0$.
\end{proof}

\section{Local rigidity of the homomorphism into the group of jets}\label{sect:loc_rigid_jet}
In this section, using the results obtained in Section \ref{sect:Lie_cohom}, we will show Proposition \ref{prop:jet_rigid} which claims  local rigidity in a weak sense of the homomorphism of the standard subgroup into the group of jets.
Let $J^r(G/P, \o)$, $r \geq 0$ be the group of $r$-jets at $\o \in G/P$.
The $C^s$-topology $(s \geq 0)$ on $\mathrm{Diff}(G/P)$ induces a topology on $J^r(G/P, \o)$ which will also be called the $C^s$-topology.
When $r \leq s$, the topology is the same as that as a Lie group, while when $s < r$, the topology is not Hausdorff. 
The statement of the following proposition is obviously weaker that that of our main theorem.

\begin{prop}\label{prop:jet_rigid}
Let $G$ be the group of orientation-preserving isometries of a  rank one symmetric space of non-compact type, $\Gamma$ a standard subgroup of $G$, and $l:P \to J^3(G/P, \o)$ the homomorphism into the group of $3$-jets at $\o \in G/P$ induced by the action of $P$ on $G/P$ by left translations.
If $\rho: \Gamma \to  J^3(G/P, \o)$ is a homomorphism $C^2$-close to $l|_\Gamma$, then there is an embedding $\iota$ of $\Gamma$ into $G$ as a standard subgroup and $h \in J^3(G/P, \o)$ such that 
\[
\rho(g) = h \circ l(\iota(g)) \circ h^{-1}
\]
for all $g \in \Gamma$.
\end{prop}

Using the local coordinate system $i \circ \exp:\n_- \to G/P$ around $\o \in G/P$ introduced in \ref{sss:loc_left_act}, the group $J^3(G/P, \o)$ can be identified with $J^3(\n_-, 0)$.
The induced homomorphism will also be denoted by $l:P \to J^3(\n_-, 0)$.

Assume the standard subgroup $\Gamma$ is generated by $a \in A$ and a lattice $\Lambda \subset N$.
Let
\[
Z = Z_{J^3(\n_-, 0)}(l(a))
\]
be the centralizer of $l(a)$ in $J^3(\n_-, 0)$.
Recall that $\Ad(a)|_{\g_r} = k^r\mathrm{Id}_{\g_r}$ for an integer $k \geq 2$.
By Proposition \ref{prop:std_fixed_pt} (ii), we see that the action of $a$ around $0 \in \n_-$ is the linear transformation corresponding to $k^{-2}\mathrm{Id}_{\g_{-2}} \oplus k^{-1} \mathrm{Id}_{\g_{-1}} \in \mathrm{GL}(\n_-)$.
By Sternberg's normalization \cite{Sternberg}, we see that an element $C^2$-close to $l(a)$ is conjugate to an element in $Z$.
So to prove Proposition \ref{prop:jet_rigid}, we may assume $\rho(a) \in Z$.

Let $\pi:J^3(\n_-, 0) \to J^1(\n_-, 0) = \mathrm{GL}(\n_-)$ be the natural projection.
Let us first consider the homomorphism $\pi \circ l: P \to \mathrm{GL}(\n_-)$.
Since
\[
\pi \circ l(a) = \left(
\begin{array}{cc}
k^{-1}\mathrm{Id}_{\g_{-1}} & 0\\
0 & k^{-2}\mathrm{Id}_{\g_{-2}}
\end{array} \right) \in \mathrm{GL}(\n_-) = \mathrm{GL}(\g_{-1} \oplus \g_{-2})
\]
and $\Ad(a)|_\n = k\mathrm{Id}_{\g_1} \oplus k^2\mathrm{Id}_{\g_2}$, it is easy to see that
\[
\pi \circ l(N) \subset \left\{ \left(
\begin{array}{cc}
\mathrm{Id}_{\g_{-1}} &*\\
0 & \mathrm{Id}_{\g_{-2}}
\end{array} \right)  \right\}.
\]
Let $H \subset \mathrm{GL}(\n_-)$ be the subgroup defined by
\[
H = \left\{
\left(
\begin{array}{cc}
 *& *\\
0 & *
\end{array} \right)\right\}.
\]

\begin{lem}\label{lem:1-jet_rigid}
If $f:\Gamma \to \mathrm{GL}(\n_-)$ is a homomorphism close to $\pi \circ l|_\Gamma$, then $f$ is conjugate to a homomorphism $f'$ such that $f'(\Gamma) \subset H$.
\end{lem}

This lemma can be shown easily by using the following theorem of Stowe:
\begin{thm}[\cite{Stowe}]\label{thm:Stowe}
Let $\Gamma$ be a finitely generated group, and $\rho$ a smooth action of $\Gamma$ on a manifold $M$ with a common fixed point $x_\o \in M$.
Then an action $C^2$-close to $\rho$ admits a common fixed point $x$ close to $x_\o$ if  the first cohomology $H^1(\Gamma, T_{x_\o}M)$ with respect to the isotropic representation of $\rho$ at $x_\o$ is vanished.
\end{thm}
\begin{proof}[Proof of Lemma \ref{lem:1-jet_rigid}]
As $\mathrm{GL}(\n_-)$ acts on $M = GL(\n_-)/H$ by left translations, a homomorphism of $\Gamma$ into $\mathrm{GL}(\n_-)$ induces an action of $\Gamma$ on $M$.
Let $\sigma$ be the the action induced by $\pi \circ l|_\Gamma$.
Since $\pi \circ l(\Gamma) \subset H$, $\sigma$ has the common fixed point $x_\o = eH \in M$.
Given a homomorphism $f:\Gamma \to \mathrm{GL}(\n_-)$ close to $\pi \circ l|_\Gamma$, the induced action of $\Gamma$ on $M$ is close to $\sigma$.

By the above observation on $\pi \circ l: P \to \mathrm{GL}(\n_-)$ , we see that, in the isotropic representation of $\sigma$ at $x_\o$, $a$ acts on $T_{x_\o}M = \mathfrak{gl}(\n_-)/\h$ by multiplication by $k^{-1}$ and $\Lambda$ acts trivially.
Thus by Lemma \ref{lem:cohom_isotro}, $H^1(\Gamma, T_{x_\o}M) = 0$.
By Theorem \ref{thm:Stowe}, the action of $\Gamma$ on $M$ induced by $f$ also admits a common fixed point.
Replacing $f$ with its conjugate, we may assume $f$ fixes $x_\o$.
This is equivalent to $f(\Gamma) \subset H$.
\end{proof}

By Lemma \ref{lem:1-jet_rigid}, to prove Proposition \ref{prop:jet_rigid} we may assume $\pi \circ \rho(\Gamma) \subset H$.
Then $\pi \circ \rho: \Gamma \to \mathrm{GL}(\n_-)$ induces homomorphisms of $\Gamma$ into $\mathrm{GL}(\g_{-1})$ and $\mathrm{GL}(\g_{-2})$.
By Lemma 2.2 of \cite{Asaoka1}, we see that 
\[
\pi \circ \rho(\Lambda) \subset \left\{ \left(
\begin{array}{cc}
\mathrm{Id}_{\g_{-1}} &*\\
0 & \mathrm{Id}_{\g_{-2}}
\end{array} \right)  \right\}.
\]
The proof of the following lemma is left to the reader:
\begin{lem}\label{lem:nilp1conn}
Let $L \subset \mathrm{GL}(n, \R) = J^1(\R^n, \o)$ be the group of upper triangular matrices with diagonal entries $1$ and $\pi^1_r: J^r(\R^n, \o) \to J^1(\R^n, \o)$ be the natural projection.
Then $(\pi^0_r)^{-1}(L) \subset  J^r(\R^n, \o)$ is a connected simply-connected nilpotent Lie group.
\end{lem}
Thus the image $\rho(\Lambda)$ is contained in a connected simply-connected nilpotent Lie group.
We will use the following:
\begin{thm}[\cite{Raghunathan}]\label{nilp_ext}
Let $N$ and $V$ be connected simply-connected nilpotent Lie groups, and $H$ a uniform subgroup of $N$.
Then any continuous homomorphism $f:H\to V$ can be extended uniquely to a continuous homomorphism $\bar{f}:N \to V$.
\end{thm}

It follows that a homomorphism $\rho:\Gamma \to J^3(\n_-, 0)$ satisfying
\[
\pi \circ \rho(\Lambda) \subset \left\{ \left(
\begin{array}{cc}
\mathrm{Id}_{\g_{-1}} &*\\
0 & \mathrm{Id}_{\g_{-2}}
\end{array} \right)  \right\}
\]
extends uniquely to a continuous homomorphism $\bar{\rho}:\langle a \rangle N \to J^3(\n_-, 0)$, where $\langle a \rangle N$ is the closure of $\Gamma$ in $AN \subset G = KAN$.
In fact, by Theorem \ref{nilp_ext}, the restriction $\rho|_\Lambda : \Lambda \to J^3(\n_-, 0)$ can be extended to a continuous homomorphism $\overline{\rho|_\Lambda}:N \to J^3(\n_-, 0)$.
Moreover, using the uniqueness of the extension, we see that $\overline{\rho|_\Lambda}$ is  a continuous extension of $\rho|_{\tilde{\Lambda}}$, where $\tilde{\Lambda} = \bigcup_{n \in \mathbb{Z}} a^{n}\Lambda a^{-n}$.
Define a map $\bar{\rho}:\langle a \rangle N \to J^3(\n_-, 0)$ by $\bar{\rho}(a^ng) = \rho(a^n)\overline{\rho|_{\Lambda}}(g)$ for $n \in \mathbb{Z}$, $g \in \tilde{\Lambda}$.
Since $\bar{\rho}$ is a continuous map which is an extension of $\rho$, $\bar{\rho}$ is also a group homomorphism.

Let $\mathfrak{j}^3(\n_-, 0)$ be the Lie algebra of $J^3(\n_-, 0)$.
Since $\Ad(a)$ on $\n_-$ is diagonal with eigenvalues $k^{-1}$ and $k^{-2}$, $\Ad(a)$ on
\[
\mathfrak{j}^3(\n_-, 0) = \bigoplus_{1 \leq r \leq 3}(S^r(\n_-^*) \otimes \n_-)
\]
is diagonal with eigenvalues $k^i$, $-1 \leq i \leq 5$.
Let $\mathfrak{j}^3(\n_-, 0)_i$ be the eigenspace corresponding to $k^i$.
\begin{lem}\label{lem:decomp_rigid}
Let $\bar{\rho}:\langle a \rangle N \to J^3(\n_-, 0)$ be a continuous homomorphism such that $\rho(a) \in Z$ is sufficiently close to $l(a)$ and $\bar{\rho}_*: \n \to \mathfrak{j}^3(\n_-, 0)$ its differentiation at $e \in \langle a \rangle N$.
Then
\[
\bar{\rho}_* (\g_i) \subset \mathfrak{j}^3(\n_-, 0)_i.
\]
In other words, $\bar{\rho}_*$ is $\mathfrak{a}$-invariant.
\end{lem}
\begin{proof}
For any $h \in Z$, $\Ad(h)$ preserves the decomposition $\mathfrak{j}^3(\n_-, 0) = \bigoplus_i \mathfrak{j}^3(\n_-, 0)_i$.
Since $\bar{\rho}:\langle a \rangle N \to J^3$ is a group homomorphism,
\[
\Ad(\bar{\rho}(a)) \circ \bar{\rho}_* = \bar{\rho}_* \circ \Ad(a).
\]
In particular, $ \bar{\rho}_*(\g_i)$ is contained in the eigenspace of $\Ad(\bar{\rho}(a))$ for eigenvalue $k^i$.
As $\bar{\rho}(a) \in Z$ is close to $l(a)$, we see that $\bar{\rho}_* (\g_i) \subset \mathfrak{j}^3(\n_-, 0)_i.$
\end{proof}

Let $\rho:\Gamma \to J^3(\n_-, 0)$ be a homomorphism $C^2$-close to $l|_\Gamma$ such that $\rho(a) \in Z$ and $\bar{\rho}:\langle a \rangle N \to J^3(\n_-, 0)$ its continuous extension.
Since $\rho|_\Lambda$ is $C^2$-close to $l|_\Lambda$, we see that
\[
\pi \circ \bar{\rho}_*: \n \to \mathfrak{j}^2(\n_-, 0)
\]
is close to $\pi\circ l_*|_\n: \n \to \mathfrak{j}^2(\n_-, 0)$, where $\pi:\mathfrak{j}^3(\n_-, 0) \to \mathfrak{j}^2(\n_-, 0)$ is the natural projection and $l_*:\p \to \mathfrak{j}^3(\n_-, 0)$ is the differentiation of $l$.
By Lemma \ref{lem:decomp_rigid}, $\bar{\rho}$ is $\mathfrak{a}$-invariant.
While $\rho$ is only $C^2$-close (not $C^3$-close) to $l$, using the  $\mathfrak{a}$-invariance, we can show that $\bar{\rho}_*:\n \to \mathfrak{j}^3(\n_-, 0)$ is close to $l_*|_\n:\n \to \mathfrak{j}^3(\n_-, 0)$:
\begin{lem}
Let $f:\n \to \mathfrak{j}^3(\n_-, 0)$ be an  $\mathfrak{a}$-invariant homomorphism of Lie algebras such that $\pi \circ f:\n \to \mathfrak{j}^2(\n_-, 0)$ is close to $\pi \circ l$.
Then $f$ is close to $l$.
\end{lem}
\begin{proof}
Put $\j_{i, j} = \j^3(\n_-, 0)_i \cap S^{j+1}(\n_-^*) \otimes \n_-$ so that $[\j_{i, j}, \j_{i', j'}] \subset \j_{i + i', j + j'}$.
Then
\[
\begin{array}{rlll}
\j^3(\n_-, 0)_1 &=& \j_{1, 0} &\oplus \j_{1, 1} \oplus \j_{1, 2},\\
\j^3(\n_-, 0)_2 &=&  &\oplus \j_{2, 1} \oplus \j_{2, 2}.
\end{array}
\]
By assumption, for $X \in \g_i$ ($i =1,2$), $f(X) \in \j^3(\n_-, 0)_i$ and its $\j_{i, j}$-component  b of $l(X)$.

We will first show that $f|_{\g_2}$ is close to $l|_{\g_2}$.
Fix $X \in \g_2$.
It suffices to show that $f(X)_{2,2}$ is close to $l(X)_{2,2}$.
Since $[\g_1, \g_2] = 0$, $[f(Y), f(X)] = 0$ for all $Y \in \g_1$.
The $\j_{3, 2}$-component of this equation is
\[
[f(Y)_{1,0}, f(X)_{2,2}] + [f(Y)_{1,1}, f(X)_{2,1}] = 0.
\]
Since $f(X)_{i, j}$ ($j \leq 1$) is close to $l(X)_{i, j}$, we see that
\[
[l(Y)_{1,0}, f(X)_{2,2}-l(X)_{2,2}]
\]
is close to $0$.
Using the fact that $[\g_{-2}, \g_1] = \g_{-1}$, we see that 
\[
l(\g_1)_{1,0}(\g_{-2}) = \g_{-1}
\]
under the identification $\j_{1,0} = \g_{-2}^* \otimes \g_{-1}$.
By $\j_{2,2} = (S^3(\g_{-1}^*)\otimes\g_1) \oplus (S^2(\g_{-1}^*)\otimes \g_{-2} \otimes \g_2)$ and this observation, we see that $ f(X)_{2,2}-l(X)_{2,2}$ is close to $0$.

It remains to show that $f|_{\g_1}$ is close to $l|_{\g_1}$.
Fix $X \in \g_1$.
It suffices to show that $f(X)_{1,2}$ is close to $l(X)_{1,2}$.
Since $[\g_1, \g_1]  \subset \g_2$ and $f|_{\g_2}$ is close to $l|_{\g_2}$, we see that
\[
[l(X)_{1,0}, f(Y)_{1,2} -l(Y)_{1,2}] + [l(Y)_{1,0}, f(X)_{1,2} -l(X)_{1,2}]
\]
is close to $0$ for all $Y \in \g_1$.
Using the fact that $[X, \g_1] = \g_{-1}$ for any $X \in \g_{-2}\setminus\{0\}$ and identifying $\j_{1,2}$ with $S^3(\g_{-1}^*) \otimes \g_{-2}$, it is not difficult to check that $ f(X)_{1,2}-l(X)_{1,2}$ is close to $0$.
\end{proof}
The next step of the proof is to show that there is $h \in Z$ such that $h\bar{\rho}(N)h^{-1} = l(N)$.
In other words,
\[
\Ad(h) \circ \bar{\rho}_*(\n) = l_*(\n).
\]
As $\bar{\rho}_*: \n \to \j^3(\n_-, 0)$ is close to $l_*|_\n$, the existence of such $h \in Z$ is an immediate consequence of the following.
\begin{lem}\label{lem:H1(np)a_to_H1(nj3)a}
The map $H^1(\n, \p)^\mathfrak{a} \to H^1(\n, \j^3(\n_-, 0))^\mathfrak{a}$ induced by $l_*: \p \to \j^3(\n_-, 0)$ is an isomorphism.
\end{lem}
\begin{proof}
Recall that we obtained the isomorphism 
\[
H^1(\n, \g)^\mathfrak{a} \to H^1(\n, \P(\n_-))^\mathfrak{a}
\]
in Corollary \ref{cor_1cohom}.

For an $AN$-module $V$ , $H^0(\n, V)^\mathfrak{a} \neq 0$ only if $V^\mathfrak{a} \neq 0$.
In general, since the Lie algebra $\n$ is generated by $\g_1$ on which $\Ad(a)$ acts by multiplication by $k$, $H^i(\n, V)^\mathfrak{a} \neq 0$ only if $\Ad(a)$ on $V$ has an eigenvalue $k^{i}$.

We see that $\p$ is an $\n$-submodule of $\g$ and that $\Ad(a)$ acts diagonally on $\g/\p$ with eigenvalues $k^{-1}$, $k^{-2}$.
Thus $H^1(\n, \p)^\mathfrak{a}$ is isomorphic to $H^1(\n, \g)^\mathfrak{a}$.

Recall that $\j^3(\n_-, 0) = \P(\n_-, 0)/ \bigoplus_{r\geq4}(S^r(\n_-^*)\otimes \n_-)$ under the identification $\P(\n_-) = \bigoplus_{r\geq0}(S^r(\n_-^*)\otimes \n_-)$, where $\P(\n_-, 0) = \bigoplus_{r\geq1}(S^r(\n_-^*)\otimes \n_-)$ is the polynomial vector fields vanishing at $0 \in V$.
Since $\Ad(a)$ acts diagonally on $\P(\n_-)/\P(\n_-, 0)$ with eigenvalues $k^{-1}$, $k^{-2}$, we see that $ H^1(\n, \P(\n_-))^\mathfrak{a}$ is isomorphic to $ H^1(\n, \P(\n_-, 0))^\mathfrak{a}$.
Moreover, as $\Ad(a)$ acts diagonally on $\bigoplus_{r\geq4}(S^r(\n_-^*)\otimes \n_-)$ with eigenvalues $k^i$ ($i \geq 2$), we see that $ H^1(\n, \P(\n_-, 0))^\mathfrak{a}$ is isomorphic to $ H^1(\n, \mathfrak{j}^3(\n_-, 0))^\mathfrak{a}$.
\end{proof}

Replacing $\rho$ with its conjugate, we may further assume that $\bar{\rho}(N) = l(N)$.
It remains to show that a continuous homomorphism $\bar{\rho}:\langle a \rangle N \to J^3(\n_-, 0)$ close to $l|_{\langle a \rangle N}$ with $\bar{\rho}(a) \in Z$ and $\bar{\rho}(N) = l(N)$ satisfies $\bar{\rho}(a) = l(a)$.
In fact, if this claim holds, the image of the homomorphism $\rho:\Gamma \to J^3(\n_-, 0)$ is contained in $l(\langle a \rangle N)$.
As $l:P \to  J^3(\n_-, 0)$ is an isomorphism onto its image, Proposition \ref{prop:jet_rigid} follows immediately.

Let  $\bar{\rho}:\langle a \rangle N \to J^3(\n_-, 0)$ be a continuous homomorphism close to $l|_{\langle a \rangle N}$ with $\bar{\rho}(a) \in Z$ and $\bar{\rho}(N) = l(N)$.
We will show the element $z_0 = \bar{\rho}(a)l(a)^{-1} \in Z$ close to the identity $e \in Z$ is in fact exactly $e$.
By the above assumption and the equation $\Ad(\bar{\rho}(a)) \circ \bar{\rho}_* = \bar{\rho}_* \circ \Ad(a)$, we see that $\Ad(z_0)$ fixes each element of $l_*(\n)$.
Thus to show $z_0 = e$, it suffices to show that $H^0(\n, \mathfrak{z}) = 0$, where $\mathfrak{z} = \mathrm{Lie}(Z)$.
Since $H^0(\n, \mathfrak{z}) = H^0(\n, \j^3(G/P, \o)^\mathfrak{a}) = H^0(\n, \j^3(G/P, \o))^\mathfrak{a}$, by the same argument as the proof of Lemma \ref{lem:H1(np)a_to_H1(nj3)a}, we can show $H^0(\n, \mathfrak{z}) = H^0(\n, \P(\n_-))^\mathfrak{a}$.
By Corollary \ref{cor_0cohom}, $H^0(\n, \mathfrak{z}) = 0$.
We finished the proof of Proposition \ref{prop:jet_rigid}.

When $G = \mathrm{Sp}(n+1, 1)$, $n \geq 2$ or $F_4^{-20}$, using $H^1(\n, \g)^\mathfrak{a} = 0$, we can show local rigidity in the strict sense.
\begin{cor}\label{cor:jet_loc_rigid}
When $G = \mathrm{Sp}(n+1, 1)$, $n \geq 2$ or $F_4^{-20}$, the homomorphism $l|_\Gamma:\Gamma \to J^3(G/P, \o)$ is $C^2$-locally rigid.
\end{cor}
\begin{proof}
By Lemma \ref{lem:H1nga} and Lemma \ref{lem:H1(np)a_to_H1(nj3)a},  we see $H^1(\n, \j^3(\n_-, 0))^\mathfrak{a} = 0$.
Thus for a continuous homomorphism $\bar{\rho}:\langle a \rangle N \to J^3(\n_-, 0)$ close to $l|_{\langle a \rangle N}$ with $\bar{\rho}(a) \in Z$, there is $h \in Z$ such that $h\bar{\rho}(g)h^{-1} = l(g)$ for $g \in N$.
By the same argument as  the proof of Proposition \ref{prop:jet_rigid}, the claim follows.
\end{proof}
\section{Local rigidity of the homomorphism into the group of formal transformations}\label{sect:loc_rigid_formal}
Let $\F(M, p)$ be the set of equivalence classes of diffeomorphisms defined around a point $p$ of a manifold $M$ and fixing $p \in M$ where two diffeomorphisms are equivalent if and only if their Taylor expansions at $p \in M$ are the same.
As $\F(M, p)$ has a natural group structure as a quotient of the group $\G(M, p)$ of germs at $p$ of diffeomorphisms, we call $\F(M, p)$ the \textit{group of formal transformations} at $p \in M$.
The goal of this section is to show the following weak local rigidity of a homomorphism into the group of formal transformations.

\begin{prop}\label{prop:formal_rigid}
Let $G$ be the group of orientation-preserving isometries of a rank one symmetric space of non-compact type, $\Gamma$ a standard subgroup of $G$, and $l:P \to \F(G/P, \o)$ the homomorphism into the group of formal transformations at $\o \in G/P$ induced by the action of $P$ on $G/P$ by left translations.
If $\rho: \Gamma \to  \F(G/P, \o)$ is a homomorphism $C^2$-close to $l|_\Gamma$, then there is an embedding $\iota$ of $\Gamma$ into $G$ as a standard subgroup and $h \in \F(G/P, \o)$ such that 
\[
\rho(g) = h \circ l(\iota(g)) \circ h^{-1}
\]
for all $g \in \Gamma$.
\end{prop}

As we proved Proposition \ref{prop:jet_rigid}, the rigidity of a homomorphism into the group of jets, Proposition \ref{prop:formal_rigid} is an immediate consequence of the following proposition.
Moreover, when $G = \mathrm{Sp}(n+1, 1)$, $n \geq 2$ or $F_4^{-20}$, by Corollary \ref{cor:jet_loc_rigid}, we obtain local rigidity of $l|_\Gamma: \Gamma \to J^3(G/P, \o)$.
Let $\pi^r:\F(M, p) \to J^r(M, p)$ $(r \geq 0)$ be the natural projection from the group of formal transformations onto the group of $r$-jets of diffeomorphisms.

\begin{prop}\label{prop:jet_to_formal}
Let $\Gamma \subset P \subset G$ be a standard subgroup and $l:P \to \F(G/P, \o)$ be the homomorphism induced by the left action of $P$ on $G/P$.
If $\rho:\Gamma \to \F(G/P, \o)$ is a group homomorphism such that $\pi^3 \circ \rho = \pi^3\circ l|_\Gamma: \Gamma \to J^3(G/P, \o)$, then $\rho, l|_\Gamma:\Gamma \to \F(G/P, \o)$ are conjugate.
\end{prop}
\begin{proof}
Since $\pi^3 \circ \rho(a) = \pi^3\circ l(a)$, by Sternberg's normalization, $\rho(a)$ and $ l(a)$ are conjugate.
So we may assume $\rho(a) = l(a)$.
Under this assumption, we will show $\rho = l|_\Gamma$.
By induction, it suffices to show that for $r \geq 3$, a group homomorphism $\rho:\Gamma \to J^{r+1}(G/P, \o)$ such that $\pi_{r+1}^r \circ \rho = \pi^r \circ l|_\Gamma$ and $\rho(a) = \pi^r \circ l(a)$ satisfies $\rho = \pi^{r+1} \circ l|_\Gamma$, where $\pi_s^r:J^s(G/P, \o) \to J^r(G/P, \o)$ $(s > r)$ denotes the natural projection from the group of $s$-jets to the group of $r$-jets.
Using Theorem \ref{nilp_ext} and Lemma \ref{lem:nilp1conn}, by the same argument as before, we see that $\rho$ extends to a continuous homomorphism $\bar{\rho}:\langle a \rangle N \to J^{r+1}(G/P, \o)$.

To prove the proposition, it suffices to show that the differential $\bar{\rho}_*:\n \to \j^{r+1}(G/P, \o)$ of $\bar{\rho}$ at $e \in \langle a \rangle N$ is equal to $(\pi^{r+1} \circ l)_*|_\n$, where $(\pi^{r+1} \circ l)_*:\p \to \j^{r+1}(G/P, \o)$ is the differential of $\pi^{r+1} \circ l: P \to J^{r+1}(G/P, \o)$.
Let $(\pi_{r+1}^r)_*$ be the differential at $e$ of $\pi_{r+1}^r:J^{r+1}(G/P, \o) \to J^r(G/P, \o)$.
As the projrctions onto $r$-jets of $\bar{\rho}$ and $\pi^{r+1} \circ l$ coincide, the image of $\bar{\rho}_* - (\pi^{r+1} \circ l)_*:\n \to \j^{r+1}(G/P, \o)$ is contained in the kernel of $(\pi_{r+1}^r)_*: \j^{r+1} (G/P, \o)\to \j^r(G/P, \o)$.
By the equations
\[
\bar{\rho}_* \circ \Ad(a)= \Ad(\rho(a)) \circ \bar{\rho}_* = \Ad(l(a)) \circ \bar{\rho}_*
\]
and $l_* \circ \Ad(a)= \Ad(l(a)) \circ l_*$, and the fact that $\Ad(l(a))$ on the kernel of $(\pi_{r+1}^r)_*$ does not have an eigenvalue $k$, we see that $\bar{\rho}_* - (\pi^{r+1} \circ l)_*|_\n$  is vanished on $\g_1$ which is the eigenspace of $\Ad(l(a))$ on $\g$ for the eigenvalue $k$.
As the Lie algebra $\n$ is generated by $\g_1$, we obtain $\bar{\rho}_* = (\pi^{r+1} \circ l)_*|_\n$.
\end{proof}
\section{Local rigidity of local actions}\label{sect:loc_rigid_loc}
Let $\G(G/P, \o)$ be the group of germs at $\o = eP \in G/P$ of diffeomorphisms in $\mathrm{Diff}(G/P, \o)$.
As $P$ is the stabilizer at $\o$ of the action of $G$ on $G/P$ by left translations, we obtain a group homomorphism of $P$ into $\G(G/P, \o)$, which will also be denoted by $l:P \to \G(G/P, \o)$.
The goal of this section is to show weak local rigidity of local actions of standard subgroups:

\begin{prop}\label{prop:locrigidgerm}
Let $G$ be the group of orientation-preserving isometries of a rank one symmetric space of non-compact type, $\Gamma$ a standard subgroup of $G$, and $l:P \to \G(G/P, \o)$ the homomorphism into the group of germs at $\o \in G/P$ of diffeomorphisms around $\o \in G/P$ induced by the action of $P$ on $G/P$ by left translations.
If $\rho: \Gamma \to  \G(G/P, \o)$ is a homomorphism $C^2$-close to $l|_\Gamma$, then there is an embedding $\iota$ of $\Gamma$ into $G$ and $h \in \G(G/P, \o)$ such that 
\[
\rho(g) = h \circ l(\iota(g)) \circ h^{-1}
\]
for all $g \in \Gamma$.
\end{prop}

By Proposition \ref{prop:formal_rigid}, to prove Proposition \ref{prop:locrigidgerm} it suffices to show the following proposition which claims that a local action close to $l|_\Gamma$ is determined by its Taylor expansions at $\o \in G/P$.
Moreover, when $G = \mathrm{Sp}(n+1, 1)$, $n \geq 2$ or $F_4^{-20}$, as we mentioned in Section \ref{sect:loc_rigid_formal}, the homomorphism of $\Gamma$ into $\F(G/P, \o)$ is locally rigid in the strict sense.
Thus using the following proposition, we see that the homomorphism of $\Gamma$ into $\G(G/P, \o)$ is also locally rigid in the strict sense.
\begin{prop}\label{prop:formal_to_local}
Let $\Gamma$ be a standard subgroup of $P$, $\rho:\Gamma \to \mathcal{G}(G/P, \o)$ a homomorphism, and $\pi:\G(G/P, \o) \to \F(G/P, \o)$ the natural projection.
If $\pi \circ \rho = \pi \circ l|_\Gamma:\Gamma \to \F(G/P, \o)$, then $\rho, l|_\Gamma \in \mathrm{Hom}(\Gamma, \mathcal{G}(G/P, \o))$ are conjugate.
\end{prop}

The remaining of this section is devoted to the proof of Proposition \ref{prop:formal_to_local}.
Since $\pi \circ \rho(a) = \pi \circ l(a) \in \F(G/P, \o)$, by Sternberg's normalization \cite{Sternberg}, we may assume $l(a) = \rho(a)$.
Put $\tilde{\Lambda} = \Gamma \cap N = \bigcup_{i \in \mathbb{Z}} a^i\Lambda a^{-i} $.
We will show $l|_{\tilde{\Lambda}} = \rho|_{\tilde{\Lambda}}$.

As $\o \in G/P$ is the common fixed point of $l(P)$, $l$ induces an action of $P$ on the complement $X = (G/P)\setminus\{\o\}$.
Using the simply transitivity of the left action of $N$ on $X$, we can define an $\n = \mathrm{Lie}(N)$-valued $l(N)$-invariant 1-form $\omega$ on $X$ as follows.
Let $x \in X$ be the fixed point of the action of $a$ on $X$.
Then under the identification $N \to X$, $g \mapsto gx$ of $N$ with $X$, we define the $\n$-valued 1-form $\omega$ on $X$ to be the pull-back of the Maurer-Cartan form on $N$.
Observe that a diffeomorphism $f$ of $X$ is contained in $l(N)$ if and only if $f^* \omega = \omega$.
Moreover, for $g_1, g_2 \in N$, the Taylor expansions at $\o G/P$ of $l(g_1)$ and $l(g_2)$ coincide if and only if $g_1 = g_2$.
So $l|_{\tilde{\Lambda}}= \rho|_{\tilde{\Lambda}}$ if and only if $\rho(\tilde{\Lambda})$ preserves $\omega$.

While the 1-from $\omega \in \Omega^1(X; \n)$ cannot be extended smoothly on $G/P$, by Proposition \ref{prop:std_fixed_pt} (iii), it is rational around $o \in G/P$ in the local coordinate $\exp \circ i: \n_- \to G/P$ around $o \in G/P$.
In particular, for any diffeomorphisms $f, g$ defined around $\o \in G/P$ fixing $\o$, if the Taylor expansions at $\o \in G/P$ of $f$ and $g$ are the same, then $f^*\omega - g^*\omega$ is a smooth 1-form around $\o \in G/P$.
Thus for each $g \in \tilde{\Lambda}$, 
\[
\Phi(g) = \rho(g)^*\omega - l(g)^*\omega
\]
is a germ of a smooth 1-form defined around $\o \in G/P$.
Since $\omega$ is $l(N)$-invariant, $\Phi(g) = \rho(g)^*\omega - \omega$.
So $\omega$ is $\rho(\tilde{\Lambda})$-invariant if and only if $\Phi(\tilde{\Lambda}) = 0$.
Observe that for $g, h \in \tilde{\Lambda}$,
\[
\Phi(gh) = \rho(h)^*(\rho(g)^*\omega - \omega) + \rho(h)^*\omega- \omega = \rho(h)^*\Phi(g) + \Phi(h).
\]
Thus $\Phi:\tilde{\Lambda} \to \Omega^1(G/P, \o; \n)$ is a cocycle, where $\Omega^1(G/P, \o; \n)$ denotes the space of germs at  $\o \in G/P$ of $\n$-valued 1-forms defined around $\o \in G/P$.
Moreover, for $g \in \tilde{\Lambda}$,
\begin{align*}
\Phi(aga^{-1}) &= l(a^{-1})^*\rho(g)^*l(a)^*\omega - \omega\\
&=  l(a^{-1})^*\rho(g)^*( \Ad(a) \circ \omega) - \omega\\
&= l(a^{-1})^*( \Ad(a) \circ \rho(g)^*\omega) - \omega\\
&= l(a^{-1})^*( \Ad(a) \circ \rho(g)^*\omega - \Ad(a) \circ \omega)\\
&=  l(a^{-1})^* \Ad(a) \circ\Phi(g).
\end{align*}
Now Proposition \ref{prop:formal_to_local} is a consequence of the following lemma.
\begin{lem}
A $\langle a \rangle$-equivariant cocycle $\Phi:\tilde{\Lambda} \to \Omega^1(G/P, \o; \n)$ is vanished.
\footnote{Since $\Omega^1(G/P, \o; \n)^{\langle a \rangle} = 0$, this is equivalent to $H^1(\tilde{\Lambda}, \Omega^1(G/P, \o; \n))^{\langle a \rangle} = 0$.}
\end{lem}
\begin{proof}
Let $\Phi:\tilde{\Lambda} \to \Omega^1(G/P, \o; \n)$ be a $\langle a \rangle$-equivariant cocycle.
Recall that the lattice $\Lambda$ has a set of generators $b_1, \dots, b_{m_1}, c_1, \dots, c_{m_2}$ such that $[b_i, c_j] = e$, $[b_i, b_j] \in \langle c_1, \dots, c_{m_2} \rangle $,  $[c_i, c_j] =e$, $ab_ia^{-1} = b_i^k$, $ac_ia^{-1} = c_i^{k^2}$.
Thus, to prove $\Phi = 0$, by the $\langle a \rangle$-equivariance, it suffices to show that $\Phi(g) = 0$ for all $g \in \tilde{\Lambda}$ with $aga^{-1} = g^k$.
Fix $g \in \tilde{\Lambda}$ with $aga^{-1} = g^k$.
Then
\[
l(a^{-1})^*\Ad(a)\circ\Phi(g) = \sum_{j = 0}^{k-1} \rho(g^j)^*\Phi(g).
\]
Recall that we have the local coordinate system $\exp \circ\; i: \n_- \to G/P$ around $\o \in G/P$ in which the differential $T_\o l(a) $ at $\o$ of $l(a)$ is of the form $k^{-1}\id_{\g_{-1}} \oplus k^{-2}\id_{\g_{-2}}$ and that of $l(g^ja)$ is of the form
\[
\left( \begin{array}{cc}
k^{-1}\id_{\g_{-1}} & u_j\\
0 & k^{-2}\id_{\g_{-2}}
\end{array} \right)
\]
for some $u_j:\g_{-2} \to \g_{-1}$.
As we fixed a local coordinate system around $\o \in G/P$, an $\mathfrak{n}$-valued 1-form around $\o \in G/P$ can be considered as an $\n_-^*\otimes\mathfrak{n}$-valued smooth function around $\o \in G/P$.
So to prove $\Phi(g) = 0$, it suffices to show that an $\n_-^*\otimes\mathfrak{n}$-valued smooth function $F$ defined around $\o$ satisfying
\[
\Ad(a) F(x) = \sum_{j = 0}^{k-1}F(\rho(g^ja)(x)) T_x \rho(g^ja)
\]
is vanished around $\o \in G/P$.
Fix a norm on $\n_-$.
There is a neighborhood $U$ of $\o \in G/P$ such that 
\begin{itemize}
\item $F$ is defined on $U$, 
\item $\rho(g^ja)x \in U$ for $j = 0, \dots, k-1$ and $x \in U$, and
\item $\|T_x \rho(g^ja)\| < k^{-1} + \epsilon$ for $j = 0, \dots, k-1$ and $x \in U$,
\end{itemize}
where $\|A\| = \sup_{v \in \n_-} \|Av\|/\|v\|$ denotes the operator norm.
Moreover, fixing a norm on $\n$, since $\Ad(a)$ on $\n$ is diagonal with eigenvalues $k, k^2$, $\|\Ad(a)^{-1}\| = k^{-1}$ with respect to the induced norm on $\mathfrak{gl}(\n)$.
Moreover, we obtain the induced norm on $\n_-^*\otimes\mathfrak{n}$.
Then
\begin{align*}
\sup_{x \in U}\|F(x)\| &= \sup_{x \in U}\|\sum_{j = 0}^{k-1}\Ad(a)^{-1}F(\rho(b^ja)(x)) T_x \rho(b^ja)\|\\
&\leq  \sum_{j = 0}^{k-1}\|\Ad(a)^{-1}\|\sup_{x \in U}\|F(\rho(b^ja)(x))\|\sup_{x \in U}\| T_x \rho(b^ja)\|\\
&\leq  \sum_{j = 0}^{k-1}k^{-1}\sup_{x \in U}\|F(x)\| (k^{-1} + \epsilon)\\
&= (k^{-1} + \epsilon) \sup_{x \in U}\|F(x)\|.
\end{align*}
It follows that $\sup_{x \in U}\|F(x)\| = 0$.
Thus $\Phi(g) = 0$.
\end{proof}

\section{Local rigidity of group actions}\label{sect:loc_rigid_act}
Let $G$ be a group of orientation-preserving isometries of a rank one symmetric space of non-compact type with an Iwasawa decomposition $G = KAN$, $P$ a minimal parabolic subgroup of $G$ containing $AN$, $l:G \to \mathrm{Diff}(G/P)$ the action by the left multiplication, $\Gamma = \langle a, \Lambda \rangle$ the standard subgroup generated by $a \in A$ and a lattice $\Lambda \subset N$ with $a\Lambda a^{-1} \subset \Lambda$.

Let $\rho$ be an action of $\Gamma$ on $G/P$ sufficiently close to $l|_\Gamma$.
Since $\Gamma \subset P$, the original action $l|_\Gamma$ admits a common fixed point $\o = P \in G/P$.
We will use Theorem \ref{thm:Stowe} to show that $\rho$ also admits a common fixed point close to $\o$.
It suffices to show that the first cohomology with respect to the isotropic representation $dl|_\Gamma: \Gamma \to \mathrm{GL}(T_\o (G/P))$ is vanished.
Under the natural identification of $T_\o (G/P)$ with $\g/\p$, the isotropic representation at $\o \in G/P$ of the left action is identified with the adjoint representation of $\Gamma$ on $\g/\p$.
By Proposition \ref{prop:vanish}, the cohomology is vanished.
Thus $\rho$ admits a common fixed point close to $\o$.

Conjugating $\rho$ by a diffeomorphism of $G/P$ which maps the common fixed point of $\rho$ to $\o$, we may assume that $\rho$ has a common fixed point $\o$.
By Proposition \ref{prop:locrigidgerm}, we may assume that for each $g\in \Gamma$, the germs at $\o \in G/P$ of $\rho(g)$ and $l(g)$ are the same.
To prove Theorem \ref{thm:main}, it remains to show the following proposition.
Moreover, when $G = \mathrm{Sp}(n+1, 1)$, $n \geq 2$ or $F_4^{-20}$, Corollary \ref{cor:main} also follows from this proposition.

\begin{prop}\label{prop:local_to_global}
Assume $G \neq \mathrm{PSL}(2, \R)$.
Let $\rho:\Gamma \to \mathrm{Diff}(G/P, \o)$ be an action of $\Gamma$ on $G/P$ with a common fixed point $\o$ whose germs at $\o$ coincides with that of $l|_\Gamma:\Gamma \to \mathrm{Diff}(G/P, \o)$.
Then $\rho, l|_\Gamma \in \mathrm{Hom}(\Gamma, \mathrm{Diff}(G/P, \o))$ are conjugate.
\end{prop}

The outline of the proof can be described as follows.
The left action of $N$ on $G/P$ has a unique common fixed point $\o$, while its action on the complement $(G/P)\setminus\{\o\}$ is simply transitive.
Thus if we fix a point $x \neq \o \in G/P$, we obtain a natural identification of $(G/P)\setminus\{\o\}$ with $N$.
Then the conjugacy around $\o \in G/P$ can be considered as a $\Gamma$-equivariant function ``at infinity'' of $N$.
Lemma \ref{lem:equiv_ext} implies that such a function can be extended $\Lambda$-equivariantly.
As $\Lambda$ is a normal subgroup of $\Gamma$, we can deduce the $\Gamma$-equivariance.
 
Let us begin with an easy lemma which gives a sufficient condition for the existence of an equivariant extension of a function.
\begin{lem}\label{lem:equiv_ext_}
Let $\Lambda$ be a group with a generating set $S$, and $X$, $Y$ manifolds on which $\Lambda$ acts smoothly.
Assume there is an open subset $U$ of $X$ such that
\begin{enumerate}
\item $X = \bigcup_{g \in \Lambda} gU$, and
\item for $g \in \Lambda$, $gU \cap U \neq \emptyset$ only if $g \in S$.
\end{enumerate}
If $f$ is a smooth map from $X$ into $Y$ such that $gf(x) = f(gx)$ for $x \in U$, then there is a unique $\Lambda$-equivariant smooth map $\tilde{f}$ from $X$ into $Y$ such that $\tilde{f} = f$ on $U$.
\end{lem}
\begin{proof}
By the assumption (i) on $U$, for any $x \in X$, there is $g \in \Lambda$ such that $gx \in U$.
Thus it suffices to show that for any $x \in X$, $\tilde{f}(x) = g^{-1}f(gx)$ does not depend on the choice of $g \in \Lambda$ with $gx \in U$.
If $g_1x, g_2x \in U$, by the assumption (ii), there is $s \in S$ with $g_2 = sg_1$.
So
\[
g_2^{-1}f(g_2x) = (sg_1)^{-1}f(sg_1x) = g_1^{-1}s^{-1}f(sg_1x) = g_1^{-1}f(g_1x).
\]
Thus the claim follows.
\end{proof}

A finitely generated group $\Lambda$ is said to have exactly one \textit{end} if the Cayley graph $\Delta = \mathrm{Cay}(\Lambda, S)$ of  $\Lambda$ with respect to a finite generating set $S$ has the following property: For any finite subgraph $F \subset \Delta$, there is a finite subgraph $F'$ containing $F$ such that the complement $\Delta \setminus F'$ is connected.
It is known that this condition does not depend on the choice of a finite generating set.

\begin{lem}\label{lem:equiv_ext}
Let $\Lambda$, $S$, $X$, $Y$, and $U$ as in Lemma \ref{lem:equiv_ext_}.
Assume further that $S$ is a finite set, $\Lambda$ has exactly one end, and the center of $\Lambda$ is infinite.
Let $f$ be a smooth map from $X$ into $Y$ such that for any $g\in\Lambda$, there is a compact subset $K_g$ of $X$ such that $gf(x) = f(gx)$ for $x \in X\setminus K_g$.
Then there is a $\Lambda$-equivariant smooth map $\tilde{f}$ from $X$ into $Y$ and a compact subset $\tilde{K}$ of $X$ such that $\tilde{f} = f$ on $X\setminus \tilde{K}$.
\end{lem}
\begin{proof}
Put $K = \bigcup_{g\in S} K_g$ so that $gf(x) = f(gx)$ for all $g \in S$ and $x \in  X\setminus K$.
By the assumption (ii) on $U$, there are at most finitely many elements $g \in \Lambda$ such that $SgU \cap K \neq \emptyset$.
As $\Lambda$ has one end, there is a finite subset $F$ of $\Lambda$ such that $SgU \cap K = \emptyset$ for $g \in \Lambda \setminus F$ and that the complement of $\mathrm{Cay}(\Lambda, S)$ for $F$ is connected.
As the center of $\Lambda$ is infinite, we may choose an element $c \in \Lambda$ in the center so that $c \not\in F$.
As $c$ commutes with any elements in $\Lambda$, $cU$ also satisfies the assumptions (i) and (ii).
By Lemma \ref{lem:equiv_ext_}, there is a unique $\Lambda$-equivariant smooth extension $\tilde{f}$ of $f|_{cU}$ to $X$.
Observe that for $g \in \Lambda$ with $\tilde{f} = f$ on $gU$, if $sg \in \Lambda \setminus F$, $s \in S$, then  $\tilde{f} = f$ on $sgU$.
Since $\tilde{f} = f$ on $cU$ and the complement of $\mathrm{Cay}(\Lambda, S)$ for $F$ is connected, we see that  $\tilde{f}$ coincides with $f$ on $\bigcup_{g \in \Lambda \setminus F} gU$.
Thus $\tilde{f} = f$ on $X\setminus \tilde{K}$ for some compact subset $\tilde{K}$ of $X$.
\end{proof}

\begin{proof}[Proof of Proposition \ref{prop:local_to_global}]
As we assume $G \neq \mathrm{PSL}(2, \R)$, the Lie group $N$ is diffeomorphic to $\R^n$ for some $n \geq 2$.
So its lattice $\Lambda$ has exactly one end.
Moreover, the center of $\Lambda$ is infinite.
As the left action of $N$ on $X = (G/P) \setminus\{\o\}$ is simply transitive, the action of $\Lambda$ on $X$ is properly discontinuous and cocompact.
So there are a finite generating set $S$ of $\Lambda$ and an open subset $U$ of $X$ satisfying the assumptions of Lemma \ref{lem:equiv_ext_}.
Since the germs at $\o$ of $l|_\Gamma$ and $\rho$ are the same, for each $g \in \Lambda$, there is a compact subset $K_g$ of $X$ such that $l(g) = \rho(g)$ on $X \setminus K_g$.
Applying Lemme \ref{lem:equiv_ext}, we obtain a $\Lambda$-equivariant smooth map $\tilde{f}:X \to X$  which is identity outside of a compact subset,  where the domain is equipped with the $\Lambda$-action induced by $l$ and the range with the action induced by $\rho$.
It remains to show that the extension $h:G/P \to G/P$ of $\tilde{f}$ by $h(\o) = \o$ is a conjugacy between $l|_\Gamma$ and $\rho$.

By the $\Lambda$-equivariance, $h$ is a covering map over $G/P$, which is diffeomorphic to the sphere $\mathrm{S}^n$, $n \geq 2$.
So $h$ is a diffeomorphism of $G/P$.
We will show the $\langle a \rangle$-equivariance of $h$.
For any $x \in X$, we can choose $g \in \Lambda$ so that $l^g(x) = l(g)(x)$ is sufficiently close to $\o$.
So we may choose $g \in \Lambda$ satisfying $h\circ l^{ag}(x) = \rho^a\circ h \circ l^g(x)$. 
By the definition of $\Gamma$, $ag^{-1}a^{-1}$ is an element of $\Lambda$.
Using the $\Lambda$-equivariance of $h$,
\[
h\circ l^a(x) = h\circ l^{ag^{-1}a^{-1}ag}(x) = \rho^{ag^{-1}a^{-1}}\circ h \circ l^{ag}(x) = \rho^{ag^{-1}}\circ h \circ l^g(x) = \rho^a \circ h(x),
\]
which shows the $\langle a \rangle$-equivariance of $h$.
So $h$ is $\Gamma$-equivariant and thus a conjugacy between $l|_\Gamma$ and $\rho$.
\end{proof}

Finally, we will show that the action $l|_\Gamma$ of $\Gamma$ on $G/P$ is not locally rigid if $G = \mathrm{SU}(n+1, 1)$, $n \geq 2$.
\begin{prop}\label{prop:nonloc_rigid}
When $G = \mathrm{SU}(n+1, 1)$, $n \geq 2$, the action $l|_\Gamma$ of a standard subgroup $\Gamma$ of $G$ on $G/P$ is not $C^2$-locally rigid.
\end{prop}
\begin{proof}
Let $l:P \to J^3(G/P, \o)$ be the homomorphism induced by the action of $P$ on $G/P$ by the left translation.
We will show that there is an automorphism $\phi$ of the group $AN$ close to the identity such that 
\begin{itemize}
\item $\phi|_A = \mathrm{id}_A$, $\phi(N) = N$, and
\item the homomorphisms $l \circ \phi$, $l|_{AN}$ of $AN$ into $J^3(G/P, \o)$ are not conjugate.
\end{itemize} 
Let $G_1$ be the group of automorphisms of $AN$ that fix $A$ and preserve $N$ and $G_2$ the subgroup of $J^3(G/P, \o)$ consisting of elements commuting with $l(A)$ and normalizing $l(N)$.
It suffices to show that the dimension of $G_1$ is larger that that of $G_2$.
It is easy to see that the Lie algebra of $G_1$ can be identified with the space $\mathrm{Der}(\n)^\mathfrak{a}$ of $\mathfrak{a}$-equivariant derivations of $\n$.
On the other hand, the Lie algebra of $G_2$ can be identified with the subalgebra of $\j^3(G/P, \o)$ consisting of elements centralizing $l(\mathfrak{a})$ and normalizing $l(\n)$.
By Lemma \ref{lem:nonlocrigid}, this subalgebra is equal to $l(\g^\mathfrak{a})$.
Since the  codimension of $l(\g^\mathfrak{a}) \subset \mathrm{Der}(\n)^\mathfrak{a}$ is equal to the dimension of $H^1(\n, \g)^\mathfrak{a} \neq 0$, the claim follows
\end{proof}

\end{document}